%% file: Revision-Eisenstein-MA.tex
\documentclass[a4paper,11pt]{amsart} 
\usepackage[T1]{fontenc}
\usepackage[utf8]{inputenc} 

\usepackage{booktabs} 
\usepackage{array} 
\usepackage{paralist} 
\usepackage{verbatim} 
\usepackage{subfig} 
\usepackage[backend=bibtex,style=alphabetic,sorting=nyvt]{biblatex}
\usepackage{amssymb}
\usepackage[dvipsnames]{xcolor}
\usepackage{mathtools}
\usepackage[scr=rsfs,cal=boondox]{mathalfa}
\usepackage{amsthm}
\usepackage{tikz-cd}
\usepackage[colorlinks, allcolors=blue]{hyperref}
\usepackage{graphicx} 
\usepackage{calrsfs}
\usepackage{enumitem}
\mathtoolsset{showonlyrefs}
\DeclareMathAlphabet{\pazocal}{OMS}{fszplm}{m}{n}

\usepackage{fancyhdr} 
\bibliography{Revision-Eisenstein_biblio.bib}
\input{Revision-commands2.tex}

\newtheoremstyle{mytheoremstyle} 
    {1.5em}                    
    {1.em}                    
    {\itshape}                   
    {}                           
    {\normalsize \bfseries}                   
    {.}                          
    {0,5em}                       
    {}  
    
\theoremstyle{mytheoremstyle}

\newtheorem{thm}{Theorem}[section]
\newtheorem*{thm*}{Theorem}
\newtheorem{cor}{Corollary}[thm]
\newtheorem*{cor*}{Corollary}
\newtheorem{lem}[thm]{Lemma}
\newtheorem{prop}[thm]{Proposition}

\newtheorem*{que*}{Question}

\theoremstyle{remark}
\newtheorem{rmk}{Remark}[section]

\numberwithin{equation}{section}

\renewenvironment{abstract}
 {\small
  \begin{center}
  \bfseries \abstractname\vspace{-.5em}\vspace{10pt}
  \end{center}
  \list{}{
    \setlength{\leftmargin}{2cm}%
    \setlength{\rightmargin}{\leftmargin}%
  }%
  \item\relax}
 {\endlist}

\begin{document}
\title{Eisenstein classes and generating series of modular symbols in $\SL_N$}
\author{ Romain Branchereau}
\date{} 
\maketitle

\begin{abstract}We define a theta lift between the homology in degree $N-1$ of a locally symmetric space associated to $\SL_N(\R)$ and the space of modular forms of weight $N$, similar to the Kudla-Millson lift in the orthogonal setting. We show that the Fourier coefficients of this lift are Poincaré duals to modular symbols associated to maximal parabolic subgroups. The constant term is a canonical cohomology classes obtained by transgressing the Euler class of a torus bundle. When $N=2$, we show that the lift surjects on the space of weight $2$ modular forms spanned by an Eisenstein series and the eigenforms with non-vanishing $L$-function.
\end{abstract}

{
  \tableofcontents
}

\section{Introduction}
\thispagestyle{empty}

\subsection{The work of Kudla-Millson} \label{intro: kudla millson} In a series of influential papers \cite{km2,km3,kmihes}, Kudla and Millson construct (in greater generality) a geometric theta lift relating the cohomology of orthogonal locally symmetric spaces $\D_\Gamma \coloneqq \Gamma \backslash \D$ of signature $(p,q)$ to holomorphic modular forms of weight $\frac{p+q}{2}$. The symmetric space $\D \simeq \SO(p,q)/\SO(p)\times \SO(q)$ is the Grassmanian of negative $q$-planes in a real quadratic space of signature $(p,q)$, and $\Gamma \subset \SO(p,q)$ is a torsion-free subgroup preserving an integral lattice in $\R^{p+q}$.
 The central ingredient of their construction is a suitable differential $q$-form $\varphi_{\KM} \in \Omega^q(\D) \otimes \Scal(\R^{p+q})$ valued in the space of Schwartz function in $\Scal(\R^{p+q})$. The theta machinery introduced by Weil \cite{weil} produces a theta kernel
\begin{align}
\Theta_{\KM}(z,\tau) \in \Omega^q(\D)^\Gamma \otimes C^\8(\HH), \qquad z \in \D, \ \quad \tau \in \HH
\end{align}
that transforms like a modular form of weight $(p+q)/2$ in $\tau$, for a certain level $\Gamma'\subseteq \SL_2(\Z)$ depending on the choice of the lattice. This theta kernel can be used to define two lifts. Let $\Mcal_{\frac{p+q}{2}}(\Gamma')$ be the space of modular forms of weight $(p+q)/2$, and $\Scal_{\frac{p+q}{2}}(\Gamma')$ be the subspace of cusp forms. On the one hand, the Petersson inner product of a cusp form $f$ with the theta kernel is a closed form, which defines a lift
\begin{align}
\Lambda_{\KM} \colon \Scal_{\frac{p+q}{2}}(\Gamma') \longrightarrow H^q(\D_\Gamma;\C).
\end{align} On the other hand, we can integrate the theta kernel on a cycle $C$ in $Z_q(\D_\Gamma;\Z)$ to define
\begin{align} \label{KMlift intro}
\Theta_{\KM} \colon H_q(\D_\Gamma;\Z) \longrightarrow \Mcal_{\frac{p+q}{2}}(\Gamma').
\end{align}
The main feature of this lift is that it has the Fourier expansion
\begin{align}
\int_{C} \Theta_{\KM}(z,\tau) = \int_Cc_0+ \sum_{n=1}^\8 (C,C_n)q^n
\end{align}
where the {\em special cycles}
\begin{align}
C_n \in Z^{\BM}_{pq-q}(\D_\Gamma;\Z)
\end{align} are locally finite cycles of codimension $q$ in $\D_{\Gamma}$, and $(C,C_n)$
is the intersection number in $\D_\Gamma$ between the two cycles $C$ and $C_n$.

In fact, Kudla and Millson constructed a more general lift between the cohomology of orthogonal locally symmetric spaces and Siegel modular forms, corresponding to the dual pair $\SO(p,q) \times \Sp_n(\R)$. They also provided an analogous lift for locally symmetric spaces attached to unitary groups. These lifts have become an important tool in arithmetic geometry, and there is an extensive literature on generalizations and applications by many authors, see for example \cite{fbtwolifts,bfinj,bmm,blmm,EY20,fmbound1,fmcoef,funkmil,G20,garcia23,KZ25,MZ26,Z24}.
Recently, Shi \cite{shi} has considered generalized special cycles on Hermitian locally symmetric manifolds, thus extending the work of Kudla and Millson to new dual pairs.
The aim of this paper is to construct an analogous theta lift on the symmetric space of $\SL_N(\R) $, whose Fourier coefficients are Poincaré duals of suitable cycles.

\subsection{Special cycles for $\SL_N(\R)$} 
We consider the symmetric spaces 
\begin{align}
X\coloneqq \GL_N(\R)^+/\SO(N), \qquad S\coloneqq \SL_N(\R)/\SO(N),
\end{align} where $\GL_N(\R)^+$ is the group of invertible matrices with positive determinant. The space $X$ can be identified with the space of real positive-definite symmetric invertible matrices, and $S \subset X$ as the subspace of matrices of determinant one. Let $V$ be the vector space $V \coloneqq \Q^N \oplus \Q^N$. Let $\langle-,-\rangle$ be the pairing defined by $\langle v,w\rangle \coloneqq v^tw$, let $B$ be the bilinear  form\begin{align}
B(\vbf_1,\vbf_2)=\langle v_1,w_2 \rangle +  \langle w_1,v_2 \rangle \qquad \vbf_i=(v_i,w_i)
\end{align} on $V$, and $Q(\vbf)=\frac{1}{2}B(\vbf,\vbf)$ the associated quadratic form of signature $(N,N)$. Throughout this paper, vectors in $V$ will be written $\vbf=(v,w)$ with $v,w \in \Q^N$. Let $\rho$ be the representation $\rho \colon \GL_N(\Q) \longrightarrow \SO(V,B)$ given by $\rho_g \vbf \coloneqq (gv,g^{-t}w)$. Let $L \subset \Z^N$ be a lattice and let $L^\ast \coloneqq \Hom(L,\Z) \supseteq L$ be its dual lattice with respect to the pairing. Let $V_\Z \coloneqq \Z^N \oplus \Z^N$ and
\begin{align}
\Lcal \coloneqq L \oplus L \subset V_\Z,
\end{align} whose dual with respect to the bilinear form $B$ is $\Lcal^\vee \coloneqq L^\ast\oplus L^\ast$. Let $D_\Lcal \coloneqq \Lcal^\vee/\Lcal$ be the discriminant of the lattice. A vector $\vbf_0 \in D_\Lcal$ defines a test function
\begin{align}
\id_{\vbf_0+\Lcal} \colon V \longrightarrow \Q,
\end{align}
and let $\chi \in \C[D_\Lcal]$ be a linear combination of such test functions. Let $\Gamma \subset \SL_N(\Z)$ be a torsion-free subgroup that preserves $\chi$, in the sense that $\chi(\rho_\gamma\vbf)=\chi(\vbf)$. (Note that the groups $\Gamma,\Gamma'$ that we will use from now on have nothing to do with the subgroups that appear in \ref{intro: kudla millson}).

We will be interested in the cohomology of the $(N^2+N-2)/2$ dimensional locally symmetric space  \begin{align}
S_{\Gamma} \coloneqq \Gamma \backslash S.
\end{align}
For each vector $\vbf \in V$ there is a natural submanifold $S_\vbf \subset S$ of codimension $N-1$, and is the image of an embedding
\begin{align}
\GL_{N-1}(\R)^+/\SO(N-1) \hooklongrightarrow S.
\end{align}
The projection $Z_{[\vbf]}$ of $S_\vbf$ in the quotient $S_{\Gamma}$ only depends on the class  of $[\vbf]$ in $\Gamma \backslash  \Lcal^\vee$. It defines a locally finite cycle
\begin{align}
Z_{[\vbf]} \in Z^{\BM}_{(N^2-N)/2}(S_\Gamma;\Z)
\end{align}
of codimension $N-1$ in the Borel-Moore homology. We consider the finite linear combination
\begin{align}
Z_n(\chi) \coloneqq \sum_{\substack{ [\vbf ] \in \Gamma \backslash  V \\ Q(\vbf)=n}} \chi(\vbf)Z_{[\vbf]} \in Z^{\BM}_{(N^2-N)/2}(S_\Gamma;\C).
\end{align}
\begin{rmk} The classical modular symbols $\{\alpha,\beta\}$ in $\HH$ are associated to two lines $\alpha, \beta \in \PP^1(\Q)$ and generate the relative homology of the modular curve $Y_0(p)$. 
These modular symbols can be generalized to $S$ in different ways. The modular symbols $Z_n(\chi)$ considered in this paper are associated to maximal parabolic subgroups. Similar modular symbols have been considered in previous work \cite{modsymb, schmidt}. On the other hand, Ash and Rudolph \cite{ashrud} consider generalized modular symbols $Z_Q$ associated to Borel subgroups and show that these modular symbols generate the relative homology of $S_\Gamma$.  We will compute the period of the theta lift on such modular symbols in \ref{evaluation on tori}.
\end{rmk}

\subsection{The Eisenstein class} Using the Mathai-Quillen formalism \cite{mq}, we construct a closed and $\GL_N(\R)^+$-invariant form $\varphi \in \Omega^N(X) \otimes \Scal(V_\R)$. We define the theta series
\begin{align}
\Theta_{\varphi, \chi}(z,\tau) \coloneqq j(h_\tau,i)^N  \sum_{\vbf \in V} \omega(h_\tau) \varphi(z,\vbf)\chi(\vbf) \in \Omega^N(X)^{\Gamma} \otimes C^\8(\HH),
\end{align}
where $h_\tau \in \SL_2(\R)$ is any matrix sending $i$ to $\tau \in \HH$ by Möbius transformations and $j$ is the standard automorphy factor. The theta series is $\Gamma$-invariant in the variable $z$, hence defines a form on $X_{\Gamma}=\Gamma \backslash X$. In $\tau$ it transforms like a modular form of weight $N$ for a congruence subgroup $\Gamma' \subset \SL_2(\Z)$ (that depends on $\chi$). 

\begin{rmk}(See \ref{relation kudla millson}.) This form $\varphi$ can also be obtained as the restriction of the Kudla-Millson form along the embedding $X \hooklongrightarrow \D$ induced by $\rho$, where $\D$ is the orthogonal locally symmetric space of signature $(N,N)$. The seesaw relating the two forms is
\[
\begin{tikzcd}
\SO(N,N)(\Q) \arrow[dash, dr] & \GL_2(\Q) \\
\GL_N(\Q) \arrow[dash]{u} \arrow[dash, ru] &\SL_2(\Q). \arrow[dash]{u}
\end{tikzcd}
\]

\end{rmk}

 Moreover, we have a diffeomorphism $X_\Gamma \simeq S_\Gamma \times \R_{>0}$, so that we can view $X_\Gamma$ as fiber bundle $p \colon X_\Gamma \longrightarrow S_\Gamma$ with fiber $\R_{>0}$. The integration along the fiber gives a form
\begin{align} \label{integral112}
E_{\varphi, \chi}(z_1,\tau,s) \coloneqq p_\ast \left ( \Theta_{\varphi, \chi}(z,\tau)u^{-s} \right )\in \Omega^{N-1}(S)^{\Gamma} \otimes C^\8(\HH),
\end{align}
where $z=(z_1,u)$ is the variable on $S \times \R_{>0}$ and we integrate over $u$. 

\begin{thm} \label{theoremA} The pushforward \eqref{integral112} converges for any $s \in \C$. At $s=0$ it defines a closed $(N-1)$-form $E_{\varphi, \chi}(z_1,\tau)$, which is invariant by $\Gamma \subset \SL_N(\Z)$ and transforms in $\tau$ like a modular form of weight $N$ and level $\Gamma' \subset \SL_2(\Z)$. In cohomology, it defines an element
\begin{align}
E_{\varphi, \chi}(z_1,\tau) = z(\chi)+\sum_{n=1}^\8 [Z_n(\chi)] q^n \in H^{N-1}(S_\Gamma;\C)\otimes \Mcal_N(\Gamma')
\end{align}
where $[Z_n(\chi)]$ denotes the Poincaré dual to $Z_n(\chi)$.
\end{thm}
The constant term $z(\chi) \in H^{N-1}(S_\Gamma)$ is a canonical Eisenstein class obtained by transgression of the Euler class, and for which explicit de Rham representatives have been constructed by Bergeron-Charollois-Garcia.

\subsection{Theta lifts} The Kudla-Millson theta correspondence has proved useful for deducing results regarding the span of the homology by special cycles. There are striking applications (amongst others) of these results due to Bergeron-Millson-Moeglin \cite{bmm} and Bergeron-Li-Millson-Moeglin \cite{blmm}. Similarly, it is natural to wonder\footnote{See \cite[Question~(a), p.~2]{mazur}} what part of the homology is generated by the modular symbols $Z_n(\chi)$.

The form $E_{\varphi, \chi}(z_1,\tau)$ can be used as a theta kernel to define two lifts. On the one hand, we have a lift
\begin{align} \label{lift117}
\quad E_{\varphi, \chi} \colon H_{N-1}(S_\Gamma;\Z)& \longrightarrow \Mcal_N(\Gamma'), \qquad E_{\varphi, \chi}(Z,\tau)\coloneqq \int_Z E_{\varphi, \chi}(z_1,\tau)
\end{align}
analogous to the Kudla-Millson lift \eqref{KMlift intro}. Note that by Theorem \ref{theoremA}, we have
\begin{align}
\int_Z E_{\varphi, \chi}(z_1,\tau)=\int_{Z} z(\chi)+(-1)^{\frac{N(N-1)^2}{2}}\sum_{n=1}^\8 (Z,Z_n(\chi))q^n,
\end{align}
where $(Z,Z_n(\chi))$ is the intersection pairing on $S_\Gamma$ between the two cycles $Z$ and $Z_n(\chi)$.
On the other hand, by taking the Petersson inner product of a cusp form with $E_{\varphi, \chi}(z_1,\tau)$ we also get an adjoint lift
\begin{align}
\Lambda_{\varphi, \chi} \colon \Scal_N(\Gamma') & \longrightarrow H^{N-1}(S_{\Gamma};\C).
\end{align}
There is a pairing
\begin{align} \label{perfect pairing intro}
 H_{N-1}(S_\Gamma;\C) \times H^{\BM}_{(N^2-N)/2}(S_\Gamma;\C) \longrightarrow \C.
\end{align}
The following theorem is analogous to  \cite[Theorem.~4.5]{kmcjm}.
\begin{thm} \label{theoremC} We have
\begin{align}
\ker(E_{\varphi, \chi}) \simeq \Span \{Z_n(\chi) \ \vert \ n \in \NN_{>0} \}^\perp \subset H_{N-1}(S_\Gamma;\C),
\end{align}
where the orthogonal complement is taken with respect to the pairing \eqref{perfect pairing intro}. Moreover, under Poincaré duality $H^{N-1}(S_\Gamma;\C) \simeq H^{\BM}_{(N^2-N)/2}(S_\Gamma;\C)$, we also have
\begin{align}
\image(\Lambda_{\varphi, \chi}) \subset \Span \{Z_n(\chi) \ \vert \ n \in \NN_{>0} \}.
\end{align}
\end{thm}
 \begin{rmk} If $\chi$ is a test function such that $\chi(-\vbf)=(-1)^{N-1}\chi(\vbf)$, then the lift $E_{\varphi, \chi}$vanishes.
 \end{rmk}

\subsection{Evaluation on tori} \label{evaluation on tori}
 Ash and Rudolph \cite{ashrud} prove that the relative homology of the Borel-Serre compactification $\overline{S}_\Gamma$ is spanned by modular symbols 
\begin{align}
Z_Q \in Z^{\BM}_{N-1}(S_\Gamma,\Z),
\end{align} that are translates of a split torus by an invertible matrix $Q \in \Mat_N(\Z)$. They are associated to Borel subgroups, whereas the modular symbols $Z_n(\chi)$ are associated to maximal parabolic subgroups. Suppose that $\Lcal=nV_\Z$ for some positive integer $n$ (it $\Lcal$ always contains such a lattice, so without loss of generality we can assume that $\Lcal$ is of that form).

\begin{thm} \label{theoremD} The integral over a modular symbol $Z_Q$ converges and we have
\begin{align}
\int_{Z_Q} E_{\varphi, \chi}(z_1,\tau)& = \frac{(-1)^{N-1}i^N}{\pi^{N}n^{N}} \sum_{\vbf_0 \in V/\Lcal}\four_2(\chi)(Q\vbf_0) \prod_{m=1}^N E_1\left (\tau,\frac{a_m\tau+b_m}{n} \right ),
\end{align}
where $\vbf_0=(a_1, \dots,a_N,b_1,\dots,b_N)$, the modular forms $E_1(\tau,\lambda_0)$ are weight one Eisenstein series, and $\four_2(\chi)$ is discrete Fourier transform of $\chi$.
\end{thm}

On the other hand, one can evaluate the lift on a torus attached to a totally real field $F$ of degree $N$. Let $\Ocal$ be its ring of integers. After fixing a $\Z$-basis $\Ocal \simeq \Z^N$,  an integral ideal $\m \subset \Ocal$ defines a lattice $L=\m$ in $\Z^N$. Its dual is $L^\ast=\m^{-1}\dfrak^{-1}$, and let $\chi_\m \in \C[D_{\Lcal}]$ be the test function
\begin{align}
\chi_\m=\id_{L \times L^\ast}=\id_{\m \times \m^{-1}\dfrak^{-1}}.
\end{align}
The choice of the basis induces an embedding $F^\times \hooklongrightarrow \GL_N(\Q)$ via the regular representation, which yields a compact cycle
\begin{align} \label{geodesic cycles}
Z_\Ocal = \Ocal^{\times,+} \backslash (F\otimes \R)^{1,+} \in Z_{N-1}(S_\Gamma;\Z).
\end{align}

\begin{thm} \label{theoremE} The evaluation of the lift  \eqref{lift117} on $Z_\Ocal$ is the diagonal restriction $\Escr_{\m}(\tau)$ of a Hilbert-Eisenstein series $\Escr_{\m}(\tau_1, \dots,\tau_N)$ of parallel weight one for a subgroup of $\SL_2(\Ocal)$. It has the Fourier expansion
\begin{align}
\Escr_{\m}(\tau)=\int_{Z_\Ocal} z(\chi_\m)+(-1)^{\frac{N(N-1)^2}{2}}\sum_{n=1}^\8 (Z_\Ocal,Z_n(\chi_\m) ) q^n.
\end{align}
\end{thm}
The constant term is a linear combination of partial zeta functions at $s=0$, and can be interpreted as a linking number in torus bundles (see \cite{bergeronsmf} for the case of a real quadratic field). 

\begin{rmk} The construction of the Eisenstein class in this paper uses the Mathai-Quillen formalism and shares some similarities with the construction of Bergeron-Charollois-Garcia in \cite{bcg,bcgcrm}.  In \cite{bcg}, inspired by the existence of a similar cocycle due to Charollois-Sczech \cite{cs2016} and the construction by Garcia \cite{garcia} of differential forms via superconnections, the authors construct a lift
\begin{align} \label{liftbcgintro}
E_{BCG} \colon H_{N-1}(S_\Gamma) \longrightarrow \Mcal_N(\Gamma'),
\end{align}
for some $\Gamma' \subset \SL_2(\Z)$. It follows from Theorem \ref{theoremD} and \cite[Theorem.~6]{bcg} that both lifts \eqref{lift117} and \eqref{liftbcgintro} are diagonal restrictions of Hilbert-Eisenstein series, when evaluated on cycles attached to totally real field. On the other hand, the computations in \cite{bcgcrm} show that the periods along modular symbols are products of weight one Eisenstein series (see \cite[Théorème.~2.10]{bcgcrm} for example). Hence, although it is not completely immediate by comparing the differential forms, this strongly suggests that the two lifts agree. To get a more precise result one should show that the form $E_{\psi,x}^{(N-1,0,N,0)}$ in \cite[Lemma.~28]{bcg} is exactly the form $E_{\varphi, \chi}$ considered in this paper, in the case where $\chi$ is the characteristic function of a non-zero torsion point $x$. Finally, in the case where $S_\Gamma=Y_0(p)$ (see below), the homology $H_1(Y_0(p),\C)$ is generated by cycles $Z_\Ocal$ as above, which are closed geodesics attached to real quadratic fields. In that case, it follows that the lift in this paper agrees with the lift of Bergeron-Charollois-Garcia.
\end{rmk}

\subsection{The case of the modular curve} \label{intro modular curve} In the special case where $N=2$, the symmetric space $S$ is the upper half-plane $\HH$. Let $p$ be a prime and let $\Lcal=pV_\Z$. For a suitable test function $\chi$ we can take $\Gamma=\Gamma'=\Gamma_0(p)$ so that $S_\Gamma = Y_0(p)= \Gamma_0(p) \backslash \HH$. Let $f_1, \dots,f_r$ be a basis of normalized newforms of $\Scal_2(\Gamma_0(p))$, and $E_2^{(p)}$ is the normalized Eisenstein series needed to span $\Mcal_2(\Gamma_0(p))$.
We define $\omega_E \coloneqq E_2^{(p)}(\tau)d\tau$ and $\omega^{\pm}_f \coloneqq \omega_{f} \pm \overline{\omega_{f}}$ where $\omega_{f}=f(\tau)d\tau$.

The theta lift specializes to
\begin{align} \label{lift118}
E_{\varphi, \chi} \colon H_{1}(Y_0(p);\C)& \longrightarrow \Mcal_2(\Gamma_0(p)).
\end{align}
We will show that the special cycles 
\begin{align}
Z_n(\chi)=T_n\{0,\8\}
\end{align} are the Hecke translates of the modular symbol $\{0,\8\}$ from $0$ to $\8$, and we deduce the following from Theorem \ref{theoremA}.
\begin{cor} For $Z \in H_1(Y_0(p);\C)$, the weight $2$ modular form $E_{\varphi, \chi}(C)$ admits the Fourier expansion
\begin{align}
E_{\varphi, \chi}(Z,\tau) = \frac{24}{p-1} \left ( \int_Z \omega_E \right )-\sum_{n=1}^\8 (Z,T_n\{0,\8\})q^n.
\end{align}
\end{cor}
\begin{rmk} The constant term is obtained by comparing the constant term in the spectral expansion below.
\end{rmk}
In particular, if we combine it with Theorem \ref{theoremE} we recover a result of Darmon-Pozzi-Vonk\cite[Theorem.~A]{dpv} that expresses the diagonal restriction of Hilbert-Eisenstein series as generating series of intersection numbers. Furthermore, following \cite{dpv} we deduce a spectral decomposition of the theta lift. 

\begin{cor} \label{corollaryH} For $Z \in H_1(Y_0(p);\C)$, the modular form $E_{\varphi, \chi}(Z)$ admits the spectral expansion
\begin{align}
    E_{\varphi, \chi}(Z) = \frac{24}{p-1} \left ( \int_Z \omega_E \right ) E_2^{(p)}- \sum_{i=1}^r \frac{L(f_i,1)}{i\pi\lVert f_i \rVert^2}\left ( \int_Z\omega_{f_i}^+\right) f_i.
\end{align}
Let $\Mcal_2^{0}(\Gamma_0(p))\subset \Mcal_2(\Gamma_0(p))$ be the subspace spanned by the eigenforms $E_2^{(p)}$, and $f_i$ for which $L(f_i,1) \neq 0$. The lift \eqref{lift118} is a surjective map
\begin{align}
E_{\varphi, \chi} \colon H_{1}(Y_0(p);\C)& \longrightarrow \Mcal^0_2(\Gamma_0(p)).
\end{align}
Moreover, the lift is Hecke equivariant in the sense that $E_{\varphi, \chi}(T_nZ)=T_nE_{\varphi, \chi}(Z)$ when $(n,p)=1$.
\end{cor}

The homology of $H_{1}(Y_0(p);\C)$ is generated by cycles $Z_\Ocal$ as in \eqref{geodesic cycles}, which are closed geodesics attached to real quadratic fields. We deduce the following from Theorem \ref{theoremE} and Corollary \ref{corollaryH}.
\begin{cor} \label{corollaryG} The space $\Mcal_2^{0}(\Gamma_0(p))$ is spanned by diagonal restrictions of Eisenstein series.
\end{cor}
A similar (and more precise) result has been obtained by Li \cite{li_restriction}, partially proving and refining a conjecture of Yang \cite{yang}.

\subsection{Acknowledgements}
I thank Pierre Charollois, Henri Darmon, Luis Garcia, Peter Xu and Marti Roset for helpful discussions around this paper and comments on earlier versions. I am grateful to Riccardo Zuffetti for pointing out an error in a previous version of Theorem \ref{theoremC}.

\section{Locally symmetric spaces and special cycles}\label{locally symmetric space}

\subsection{Locally symmetric space of $\SL_N(\R)$} Let $X$ be the space of real and positive-definite quadratic forms on $\R^N$. We view an element $z \in X$ as a positive definite symmetric matrix, on which the group $\GL_N(\R)^+$ acts transitively by sending $z$ to $gzg^t$. The stabilizer of $z=\id_N$ is $K=\SO(N)$ so that we can identify
\begin{align}
    X \simeq \GL_N(\R)^+/\SO(N).
\end{align} It is a Riemannian manifold of dimension $\frac{N^2+N}{2}$.
We also have the space $S \subset X$ of symmetric matrices of determinant one. Similarly, it can be identified with a homogeneous space
\begin{align}
    S \simeq \SL_N(\R)/\SO(N),
\end{align}
and the dimension of $S$ is $\frac{N^2+N-2}{2}$. The action of $\GL_N(\R)^+$ on $X$ restricts to an action of $\SL_N(\R)$ on $S$. 

Let $\rho$ denote the action of $\GL_N(\Q)$ on $V$ defined by 
\begin{align}
   \vbf=(v,w) \longmapsto \rho_g\vbf = (gv,g^{-t}w).
\end{align}
It preserves the quadratic form $Q(\vbf) =v^tw$ and gives a representation $\rho \colon \GL_N(\R)^+ \longrightarrow \SO(N,N)$.

\subsection{Special cycles in $S$}
For $\vbf=(v,w) \in V_\R$ we define the subset
\begin{align}
    X_{\vbf} \coloneqq \left \{ \left . z \in X \right \vert \ v=zw \ \right \} \subset X.
\end{align}
A vector $\vbf=(v,w)$ such that $v$ and $w$ are both nonzero will be called a {\em regular vector}, and a {\em singular vector} otherwise.

\begin{prop} \label{properties of Xv} The submanifold $X_{\vbf}$ satisfies the following properties.
\begin{enumerate} 
\item \begin{itemize} \item[$-$] If $Q(\vbf)>0$, then $X_{\vbf}$ is a submanifold of codimension $N$.
\item[$-$] If $\vbf=(0,0)$, then $X_\vbf=X$.
\item[$-$] In all other case, $X_{\vbf}$ is empty.
\end{itemize} 
\item For any $g \in G$ we have 
\begin{align}
    gX_{\vbf}=X_{\rho_g\vbf}
\end{align}
\item It is invariant by scalar multiplication {\em i.e} $X_{\lambda \vbf}=X_{\vbf}$.
\end{enumerate}
\end{prop}
\begin{proof} If $Q(\vbf)>0$, then $X_\vbf$ is the zero locus of the smooth function $z \longrightarrow zw-v$, which is of rank $N$ (since $w$ is nonzero). It is clear that $X_{(0,0)}=X$.  If $\vbf$ is a nonzero vector with $v=0$ or $w=0$, then it is clear that $X_\vbf$ is empty. Finally, if $\vbf$ is regular and $X_{\vbf}$ is non-empty, then $v=zw$ for some $z \in X_\vbf$. This implies that $0<w^tzw=w^tv=Q(\vbf)$. This proves the first point and the two other are obvious.
\end{proof}
We can identify $X \simeq S \times \R_{>0}$ by the diffeomorphism
\begin{align} \label{bundleiso}
    X  \longrightarrow S \times \R_{>0}, \qquad z \longmapsto \left (p(z), \det(z)^{-\frac{1}{N}} \right ),
\end{align}
where the projection $p \colon X \longrightarrow S$ onto $S$ is defined by $p(z)=\det(z)^{-\frac{1}{N}}z$. The inverse of this map is 
\begin{align}
    S \times \R_{>0} \longrightarrow X, \qquad (z_1,u) \longmapsto z=z_1u.
\end{align}
Hence, if we identify $\R$ with $\R_{>0}$ via the exponential map, then we can view $X$ as a real vector bundle of rank one over $S$, which is $\GL_N(\R)^+$-equivariant in the sense that
\begin{align} \label{equivariance1}
    p(gz)=gp(z).
\end{align}
Let
\begin{align}
    S_{\vbf} \coloneqq p(X_{\vbf}) \subset S
\end{align}
be the image of $X_\vbf$ in $S$. We have
\begin{align} \label{equivariance2}
    g_1S_{\vbf}=S_{\rho_{g_1}\vbf}
\end{align}
for $g_1 \in \SL_N(\R)$.
\begin{rmk}
    The cycle $S_\vbf$ only depends on the lines spanned by $v$ and $w$.
\end{rmk}
We identify $X$ with the bundle $S \times \R_{>0}$ over $S$. A pair $(z_1,u) \in S \times \R_{>0}$ lies in $X_\vbf$ exactly when $v=uz_1w$. By taking the length on both sides we find that $u$ is uniquely determined by $z_1$. In other words, if we restrict the bundle $S \times \R_{>0}$ to $S_\vbf$, we can view $X_\vbf \subset S_\vbf \times \R_{>0}$ as the image of the section
\begin{align} \label{sectionofbund}
    S_\vbf \longrightarrow S_\vbf \times \R_{>0}, \qquad z_1 \longmapsto \left (z_1,\frac{\lVert v \rVert}{\lVert z_1w \rVert} \right ).
\end{align}
We deduce the following.
\begin{prop} \label{diffeomorphism}
Let $\vbf$ be a regular vector with $Q(\vbf)>0$. The restriction of the projection $p$ to $X_{\vbf}$ is a diffeomorphism onto $S_\vbf$. Hence, the image $S_{\vbf}$ is a closed submanifold of $S$ of codimension $N-1$.   
\end{prop}
The other properties of Proposition \ref{properties of Xv} hold for $S_\vbf$.

Let $\vbf$ be a regular vector with $Q(\vbf)>0$. We can also view the submanifold $S_\vbf$ as an embedded copy of a smaller symmetric space $X_{N-1} \coloneqq \GL_{N-1}(\R)^+/\SO(N-1)$ inside $X$. The setting is similar to \cite{modsymb, schmidt}.

First, when $\vbf=(e_1,e_1)$ with $e_1=(1,0,\dots,0)^t$ the standard basis vector, the submanifold $X_\vbf$ consists of matrices of the form
\begin{align}
    z=\begin{pmatrix}
        1 & 0 \\ 0 & z'
    \end{pmatrix}
\end{align}
where $z' \in X_{N-1}$. Hence, the submanifold $X_\vbf$ is the image of the embedding
\begin{align} \label{jemb}
    j \colon X_{N-1} \hooklongrightarrow X, \qquad
    h\SO(N-1)  \longmapsto \begin{pmatrix}1 & 0 \\ 0 & h
    \end{pmatrix}\SO(N).
\end{align}
For general $\vbf=(v,w)$, let us write $v=\begin{pmatrix} v_1 \\ \vud \end{pmatrix} \in \R \times \R^{N-1}$ and $w=\begin{pmatrix} w_1 \\ \wud \end{pmatrix}$. First, we suppose that $w_1>0$. Let 
\begin{align}
    m_{\vbf} \coloneqq \begin{pmatrix} 
    v_1/\sqrt{n} & -\wud^t/w_1 \\
 \vud/\sqrt{n} & \id_{N-1} \\
    \end{pmatrix} \in \GL_N(\R)^+
\end{align}
where $n=Q(\vbf)$, and define the embedding $j_\vbf \colon X_{N-1} \hooklongrightarrow X$ by
\begin{align} \label{jemb2}
j_{\vbf}(h) \coloneqq m_\vbf \begin{pmatrix}1 & 0 \\ 0 & h\end{pmatrix}.
\end{align}
\begin{rmk}
Note that the inverse of $m_\vbf$ is 
\begin{align} \label{mvprop}
    m_\vbf^{-1}=\begin{pmatrix} w_1/\sqrt{n} & \wud^t/\sqrt{n} \\
    -\vud w_1/n & \id_{N-1}-\vud \wud^t/n
    \end{pmatrix},
\end{align}
so that $m_\vbf e_1= v/\sqrt{n}$ and $m_\vbf^{-t}e_1=w/\sqrt{n}$. Hence, the matrix $m_\vbf$ maps the line spanned by $e_1$ to the line spanned by $v$, and whose dual maps to the line spanned by $w$.
\end{rmk}

\begin{prop} When $w_1$ is positive, the submanifold $S_\vbf$ is the image of \begin{align}
    p \circ j_\vbf \colon X_{N-1} \hooklongrightarrow S, 
\end{align}
where $p$ is the projection from $X$ to $S$.
\end{prop}
\begin{proof} Recall that for any matrix $g \in \GL_N(\R)^+$ we have $gX_{(v,w)}=X_{(gv,g^{-t}w)}$. Let
\begin{align}
    g \coloneqq  \begin{pmatrix}
    1 & -\wud^t/w_1 \\
     0 & \id_{N-1}
    \end{pmatrix}
\end{align}
be the matrix such that $g^{-t}w_1e_1=w$. If $z \in X_{(v',w_1e_1)}$, then $gz \in X_{(gv',w)}=X_{(v,w)}$ where 
\begin{align}
v' \coloneqq g^{-1}v= \begin{pmatrix} n/w_1\\ \vud\end{pmatrix} \in \R \times \R^{N-1}.
\end{align}

For $\vbf'=(v',w_1e_1)$, with $v'$ as above, the submanifold $X_{\vbf'}$ consists of positive definite invertible matrices of the form
\begin{align}
    z= \begin{pmatrix}
        n/w_1^2 & \vud^t/w_1 \\
        \vud/w_1 & z'
    \end{pmatrix}
\end{align}
for some $(N-1) \times (N-1)$-matrix $z'$. In fact, it is the image of the embedding
\begin{align}
    j_{\vbf'} \colon X_{N-1} \hooklongrightarrow X, \qquad h\SO(N-1) \longmapsto j_{\vbf'}(h)\SO(N)
\end{align}
where $j_{\vbf'}$ is the embedding \eqref{jemb2} and for $\vbf'=(v',w_1e_1)$ the matrix $m_{\vbf'}$ is
\begin{align}
    m_{\vbf'} & = \begin{pmatrix}
        \sqrt{n}/w_1 & 0 \\
        \vud/\sqrt{n} & \id_{N-1}
    \end{pmatrix} \in \GL_N(\R)^+.
\end{align}
Indeed, we see that
\begin{align}
 j_{\vbf'}(h)j_{\vbf'}(h)^t & =   m_{\vbf'} \begin{pmatrix}
     1 & 0 \\ 0 & hh^t
 \end{pmatrix} (m_{\vbf'})^t = \begin{pmatrix}
     n/w_1^2 & \vud^t/w_1 \\ \vud/w_1 & \vud \vud^t/n+hh^t
 \end{pmatrix},
\end{align}
where $ \vud \vud^t/n+hh^t$ is positive-definite and invertible.

So $X_{\vbf}=gX_{\vbf'}$ is obtained by translating by $g$ and is the image of the embedding
\begin{align}
    j_{\vbf} \colon X_{N-1} \hooklongrightarrow X, \qquad h \longmapsto gj'_\vbf(h)=m_\vbf\begin{pmatrix}
        1 & 0 \\ 0 & h
    \end{pmatrix}
\end{align}
where the matrix $m_\vbf$ is
\begin{align}
    m_{\vbf} = gm'_\vbf=\begin{pmatrix} 
    v_1/\sqrt{n} & -\wud^t/w_1 \\
 \vud/\sqrt{n} & \id_{N-1} \\
    \end{pmatrix} \in \GL_N(\R)^+.
\end{align}
Hence, the submanifold $S_\vbf$ is the image of
\begin{align}
    p \circ j_\vbf \colon X_{N-1} \hooklongrightarrow S. 
\end{align}
\end{proof}
Note that $\det(m_\vbf)=\det(m_{\vbf'})=\sqrt{n}/\vert w_1 \vert$ so that
\begin{align}
p \circ j_{\vbf}(h)=\left (\frac{\vert w_1 \vert}{\sqrt{n}\det(h)}\right )^\frac{1}{N} m_\vbf\begin{pmatrix}
        1 & 0 \\ 0 & h
    \end{pmatrix}.
\end{align}
\begin{rmk}
    If $w_1$ is nonpositive, then one can compose the embedding with a permutation matrix that will bring one of the positive $w_i$'s in the first position. Note that at least one entry of $w$ must be nonzero (otherwise $X_\vbf$ is empty), and if all entries are negative we can replace $X_\vbf$ by $X_{-\vbf}=X_\vbf$. The submanifold $X_\vbf$ is then the translate by a matrix $m_\vbf$ of the embedding of the form
    \begin{align}
    j \colon X_{N-1} \hooklongrightarrow X, \qquad
    \begin{pmatrix}
        a & b \\ c & d
    \end{pmatrix}\SO(N-1)  \longmapsto \begin{pmatrix}a & 0 & b \\ 0 & 1 & 0 \\ c & 0 & d
    \end{pmatrix}\SO(N).
\end{align}
\end{rmk}

\subsection{Quotients} Let $L \subset \Z^{N}$ be a lattice, and $\Lcal=L \oplus L \subseteq V_\Z$. Let $\Lcal^\vee$ be the dual with respect to the bilinear form
\begin{align}
B(\vbf_1,\vbf_2)=\langle v_1,w_2 \rangle +  \langle w_1,v_2 \rangle
\end{align} on $V$. Let $Q(\vbf)=\frac{1}{2}B(\vbf,\vbf)$ be the associated quadratic form. As in the introduction we fix a test function $\chi \in \C[D_\Lcal]$ that is preserved by a torsion-free subgroup $\Gamma$ of $\SL_N(\Z)$. We view $\chi$ as a linear combination of characteristic functions $\id_{\vbf_0+\Lcal}$ on $V$. The quotient
\begin{align}
    X_\Gamma \coloneqq \Gamma \backslash X
\end{align}
is a manifold. The group $\Gamma$ only acts on $S$ so that we have $X_\Gamma \simeq \R_{>0} \times S_\Gamma$ where
\begin{align}
    S_\Gamma \coloneqq \Gamma \backslash S.
\end{align}

\begin{prop} \label{finiterep} For any positive integer $n$ the set 
    \begin{align}
    \left \{ \biggl . \vbf \in \Gamma \backslash \Lcal^\vee \ \biggr \vert \ Q(\vbf)=n \right \}    
    \end{align} is finite.
\end{prop}
\begin{proof}  
Since $\Lcal \subseteq V_\Z$, the proposition follows from the finiteness of
 \begin{align}
    \left \{ \biggl . \vbf \in  \SL_N(\Z) \backslash V_\Z \ \biggr  \vert \ Q(\vbf)=n \right \}.    
    \end{align}
By the Hermite normal form, we can find $\gamma \in \SL_N(\Z)$ such that 
\begin{align}
    \gamma v = \begin{pmatrix}
        D \\ \Obf_{N-1}
    \end{pmatrix} \in De_1,
\end{align} where $D=\gcd(v_1,v_2,\dots,v_N)$. Thus, every orbit can be reduced to
\begin{align}
    \SL_N(\Z)(v,w)=\SL_N(\Z) \left (\begin{pmatrix}
        D \\ \Obf_{N-1}
    \end{pmatrix},\begin{pmatrix}
        w'_1 \\ \wud'
    \end{pmatrix} \right ).
\end{align}
We have $Dw'_1=v^tw=n$, so $w'_1$ must be one of the divisors of $n$.

Reducing $w'_2$ modulo $ w'_1$ we can write $w'_2=q w'_1+r_2$ for some $r_2 \in \Z/w'_1 \Z$. Then, for
\begin{align}
    \gamma=\begin{pmatrix}
        1 & q & 0 \\
        0 & 1 & 0 \\
        0 & 0 & \id_{N-2}
    \end{pmatrix} \in \SL_N(\Z)
\end{align} we have
\begin{align} \label{representatives}
    \rho_{\gamma}\left (\begin{pmatrix}
        D \\ 0 \\ \Obf_{N-2}
    \end{pmatrix},\begin{pmatrix}
        w_1' \\ w_2' \\ \vdots
    \end{pmatrix} \right )= \left ( \begin{pmatrix}
        D \\ 0 \\ \Obf_{N-2}
    \end{pmatrix},\begin{pmatrix}
        w_1' \\ r_2 \\ \vdots
    \end{pmatrix} \right ).
\end{align}
Recall that $\rho_\gamma$ acts by $\gamma^{-t}$ in the second factor. Reducing the remaining entries $w_3', \dots, w_N'$ modulo $w'_1$, we see that every double coset can be reduced to
\begin{align}
    \SL_N(\Z)(v,w)=\SL_N(\Z) \left (\begin{pmatrix}
        D \\ \Obf_{N-1}
    \end{pmatrix},\begin{pmatrix}
        w'_1 \\ \underline{r}
    \end{pmatrix} \right )
\end{align}
where $w_1'$ is a divisor of $Q(\vbf)=n$ and $\underline{r} \in (\Z/ w'_1\Z)^{N-1}$.
\end{proof}

\subsubsection{Special cycles in the quotient.} The stabilizer of $\vbf$ under the action $\rho$ is the subgroup
\begin{align}
   \Gamma_\vbf \coloneqq \left \{ \left . \gamma  \in \Gamma \ \right \vert \ \gamma v=v, \ w^t \gamma=w^t\right \}.  
 \end{align}
Recall from \eqref{mvprop} that $m_\vbf$ satisfies $m_\vbf e_1= v/\sqrt{n}$ and $m_\vbf^{-t}e_1=w/\sqrt{n}$. If $\gamma \in \Gamma_\vbf$ then $(m_\vbf^{-1} \gamma m_\vbf)e_1=e_1$ and $e_1^t(m_\vbf^{-1} \gamma m_\vbf)=e_1^t$. Thus, we have
\begin{align}
    m_\vbf^{-1} \gamma m_\vbf=\begin{pmatrix}
        1 & 0 \\ 0 & \gamma'
    \end{pmatrix}=j(\gamma')
\end{align}
for some $\gamma' \in \SL_{N-1}(\R)$. Let $\Gamtil_\vbf \subset \SL_{N-1}(\R)$ be the subgroup such that
\begin{align}
\Gamma_\vbf \coloneqq m_\vbf j(\Gamtil_\vbf)m_\vbf^{-1}.
\end{align}
The embedding $p \circ j_\vbf$ passes to the embedding of the quotient
 \begin{align} \label{jvembls}
     p \circ j_\vbf \colon \Gamtil_\vbf \backslash X_{N-1} \hooklongrightarrow \Gamma_{\vbf} \backslash S, \qquad
     \Gamtil_\vbf z \longmapsto \Gamma_\vbf j_\vbf(z)=\Gamma_\vbf m_\vbf j(z)
 \end{align}
 with image $\Gamma_{\vbf} \backslash S_\vbf$. 
\begin{rmk}
    Although the coefficients of $m_\vbf$ contain $\sqrt{n}$, we have
    \begin{align} \label{conjugation}
m_\vbf j(\gamma) m_\vbf^{-1} &= \begin{pmatrix} v_1/\sqrt{n} & -\wud^t/w_1 \\
    \vud/\sqrt{n} & \id_{N-1}
    \end{pmatrix} \begin{pmatrix}
        1 & 0 \\
        0 & \gamma \\
    \end{pmatrix} \begin{pmatrix} w_1/\sqrt{n} & \wud^t/\sqrt{n} \\
    -\vud w_1/n & \id_{N-1}-\vud \wud^t/n
    \end{pmatrix} \nonumber \\
& =\begin{pmatrix}
        1 & 0 \\
        0 & \gamma \\
    \end{pmatrix}+\frac{1}{n} \begin{pmatrix}
    \wud^t(\gamma\vud-\vud) & (\wud^t-\wud^t\gamma)v_1 \\
    (\vud-\gamma\vud)w_1 & (\vud-\gamma \vud)\wud^t
\end{pmatrix}.
\end{align}
Hence, if $\Gamma$ is contained in $\SL_N(\Z)$ then the subgroup $\Gamtil_\vbf$ is a congruence subgroup of $\SL_N(\Z)$.
\end{rmk}

Let us denote by $Z_{[\vbf]}$ the image of the map 
\begin{align} \jov_\vbf \colon \Gamtil_\vbf \backslash X_{N-1} & \hooklongrightarrow \Gamma_\vbf \backslash S \longrightarrow S_\Gamma
\end{align}
obtained by composing the map $p \circ j_\vbf$ in \eqref{jvembls} with the projection of $ \Gamma_\vbf \backslash S$ onto $S_\Gamma$. This is the map of \cite[p.~116]{modsymb} and is proper. It defines a locally finite cycle
\begin{align}
    Z_{[\vbf]} \in Z^{\BM}_{\frac{N^2-N}{2}}\left (S_\Gamma ;\Z \right )
\end{align}
of dimension $\frac{N^2-N}{2}$ (and codimension $N-1$) that depends only the equivalence class $\vbf \in \Gamma\backslash  \Lcal^\vee/\Z^\times$. We then define for $n \in \NN_{>0}$
\begin{align}
 Z_n(\chi) \coloneqq \sum_{\substack{[\vbf] \in \Gamma \backslash  V \\ Q(\vbf)=n}} \chi(\vbf)Z_{[\vbf]} \in Z^{\BM}_{\frac{N^2-N}{2}}\left (S_\Gamma ;  \C \right ).
\end{align}
The sum is finite by Proposition \ref{finiterep} since $\chi$ is supported on $\Lcal^\vee$.

\subsubsection{Orientability.} The submanifold $X_\vbf$ is the zero locus of the function
\begin{align}
    f_\vbf \colon X  \longrightarrow \R^N, \qquad
    z \longmapsto zw-v.
\end{align}
The differential is $d_zf_\vbf(A)=Aw$ and the function is regular when $\vbf$ is regular (since it implies $w \neq 0$). At a point $z \in X_\vbf$ the tangent space splits as
\begin{align}
    T_zX \simeq T_zX_\vbf \oplus N_zX_\vbf
\end{align}
where $NX_\vbf$ is the normal bundle. The kernel of the differential is $T_zX_\vbf$ so that it induces an isomorphism $N_zX_\vbf \simeq T_0\R^N \simeq \R^N$. Hence, an orientation $o(\R^N)$ of $\R^N$ determines an orientation $o(N_zX_\vbf)$ of the normal bundle. So after fixing an orientation $o(T_zX)$ of $X$ this determines an orientation of $X_\vbf$ by the rule
\begin{align}
    o(T_zX)=o(T_zX_\vbf)\wedge o(N_zX_\vbf).
\end{align}
The orientation of $S_\vbf$ then comes from the diffeomorphism $X_\vbf\simeq S_\vbf$. 

\section{Thom forms and the Mathai-Quillen formalism} \label{sec:mq}

Let $K=\SO(N)$ be the maximal compact subgroup of $\GL_N(\R)^+$. Let $\g=\p \oplus \kfrak$ be a Cartan decomposition of the Lie algebra $\g\simeq \Mat_N(\R)$ of $\GL_N(\R)^+$, where $\kfrak=\so(\R^{N})$ is the space of skew-symmetric matrices, and $\p$ is the space of symmetric matrices. We consider the rank $N$ vector bundle $E=\GL_N(\R)^+ \times_K \R^{N}$. It consists of pairs $[g,v]$ with the equivalence relation $[g,v]=[gk,k^{-1}v]$ and the projection map $E \rightarrow X$ is $[g,v] \mapsto gK$. We view it as a metric bundle with the metric $v^tv=\lVert v \rVert^2$. The bundle is $\GL_N(\R)^+$-equivariant, where $\GL_N(\R)^+$ acts on $E$ by $g[g',v]=[gg',v]$. If $\Gamma \subset \SL_N(\Z)$ is a discrete subgroup as previously, then we have a real vector bundle $E_\Gamma \coloneqq \Gamma \backslash E$ over $X_\Gamma$.

\subsection{Thom isomorphism} \label{subsec: thom iso}
Let $\Omega_{\cv}^i(E_\Gamma)$ be the space of $i$-forms with {\em compact vertical} support {\em i.e.}, with support contained in the disk bundle $DE_\Gamma \subset E_\Gamma$. The cohomology of this complex is the relative cohomology $H^i(E_\Gamma,E_\Gamma \smallsetminus DE_\Gamma)$. 
After fixing an orientation of the fibers, the integration along the fibers of the bundle $E_\Gamma$ induces the Thom isomorphism
\begin{align}
    H^i(E_\Gamma,E_\Gamma \smallsetminus DE_\Gamma) \longrightarrow H^{i-N}(X_\Gamma).
\end{align}
The preimage of $1 \in H^0(X_\Gamma)\simeq \Z$ is a class
\begin{align}
    \Th(E_\Gamma) \in H^N(E_\Gamma,E_\Gamma \smallsetminus DE_\Gamma)
\end{align}
called the Thom class. It can also be seen as a Poincaré dual of the zero section $(E_\Gamma)_0$ in $E_\Gamma$. It is represented by a closed differential form $U \in \Omega^{N}_\rd(E_\Gamma)$ on $E_\Gamma$, that has integral $1$ along the fibers. The Mathai-Quillen formalism is a canonical way to produce such a representative, that will depend on the choice of a connection on the bundle that is compatible with the metric.

\subsection{Mathai-Quillen construction} \label{mqconstruction}

Let $\Scal(\R^{N})$ be the space of Schwartz functions on $\R^N$. By composition with the diffeomorphism 
\begin{align}
    h \colon D \longrightarrow \R^{N}, \qquad
    w  \longmapsto \frac{w}{\sqrt{1-\lVert w \rVert^2}}
\end{align}
from the unit disk $D \subset \R^{N}$ onto $\R^{N}$ we obtain a map from $\Scal(\R^{N})$ to the space of smooth functions on $\R^{N}$ supported on $D$.

Let $\Omega^i_{\rd}(E_\Gamma) \subset \Omega^i(E_\Gamma)$ be the complex of differential forms that are rapidly decreasing in the fiber of $E_\Gamma$. The pullback by $h$ induces a map $\Omega^i_{\rd}(E_\Gamma) \longrightarrow \Omega^i_{\cv}(E_\Gamma)$. 
In this section we explain how Mathai and Quillen construct a differential form 
\begin{align}
    U \in \Omega_{\rd}^{N}(E)^{\GL_N(\R)^+}
\end{align}
that is closed, $\GL_N(\R)^+$-invariant and with integral $1$ along the fiber. In particular it will be $\Gamma$-invariant and descends to a form
\begin{align}
    U \in \Omega_{\rd}^{N}(E_\Gamma)
\end{align}
that represents the Thom class.

\subsubsection{Some operations on vector bundles.} The tautological bundle $E_ {\taut}$ over $E$ is the pullback of $E$ along the projection map $E \rightarrow X$. More concretely, we have 
\begin{align}
E_ {\taut}=\GL_N(\R)^+ \times_K (\R^{N} \times \R^{N})
\end{align} where the equivalence relation is $[g,v,w]=[gk,k^{-1}v,k^{-1}w]$. Let us consider the exterior product $\bigwedge^jE_{\taut}=(\GL_N(\R)^+ \times \R^{N}) \times_K \bigwedge^j \R^{N}$ of the tautological bundle over $E$, where $K$ acts by $(gk,k^{-1}v)$ in the left factor and by $k^{-1}(v_1 \wedge \cdots \wedge v_j)=(k^{-1}v_1) \wedge \cdots \wedge (k^{-1}v_j)$ in the right factor. We also define 
\begin{align}
\bigwedge E_{\taut}=(\GL_N(\R)^+ \times \R^{N}) \times_K \bigwedge \R^{N} = \oplus_{j=0}^N \bigwedge^jE_{\taut}.
\end{align} A differential $i$-form on $E$ with values in  $\bigwedge^j E_{\taut}$ is an element in $\Omega^i(E,\bigwedge^j E_{\taut})$. It can also be seen as  a basic form ($K$-invariant and trivial on vertical vectors, see \cite[Proposition~1.9]{berline}) on $\GL_N(\R)^+ \times \R^{N}$ with values in $\bigwedge^j \R^N$. 
Let us define
\begin{align}
    \Omega^{i,j} \coloneqq\Omega^i(E,\bigwedge^j E_{\taut}) \simeq \left [ \Omega^i(\GL_N(\R)^+ \times \R^{N}) \otimes  \sideset{}{^{j}}\bigwedge \R^{N} \right ]_{\bas}
\end{align}
and $\Omega^{i,\bullet} \coloneqq \sum_{j=0}^N \Omega^{i,j}$. We define $\Omega^{\bullet,j}$ and $\Omega^{\bullet,\bullet}$ similarly. We can equip $\Omega^{\bullet,\bullet}$ with the multiplication
\begin{align}
   \Omega^{i,j} \times \Omega^{k,l} & \longrightarrow \Omega^{i+k,j+l}
\end{align}
defined by
\begin{align}
 (\omega \otimes \nu) \wedge (\omega' \otimes \nu') \coloneqq (-1)^{jk}(\omega \wedge \omega') \otimes (\nu \wedge \nu').
\end{align}
This makes $\Omega^{\bullet,\bullet}$ an associative bigraded algebra over the ring of differential forms on $E$. We extend the pairing $\langle-,-\rangle$ to a pairing on $\wedge \R^{N}$ by defining
\begin{align}
    \langle v_1\wedge \cdots \wedge v_j,v'_1 \wedge \cdots \wedge v_l' \rangle= \begin{cases}
        0 & \textrm{if} \ j \neq l \\
        \det(v_a^tv'_b)_{a,b} & \textrm{if} \ j=l.
    \end{cases}
\end{align}
With this bilinear form the splitting $\wedge \R^{N}= \oplus_{j=0}^N \wedge^j \R^{N}$ is orthogonal. Furthermore, since $K$ preserves this bilinear form, it induces a bilinear form on the bundle. The algebra structure allows us to define the exponential map
\begin{align}
    \exp \colon \Omega^{\bullet,\bullet}  \longrightarrow \Omega^{\bullet,\bullet}, \qquad
    \omega \otimes \nu \longmapsto \sum_{k=0}^N \frac{1}{k!} (\omega \otimes \nu)^{k}. 
\end{align}

Let $e_1, \dots, e_N$ be the standard basis of $\R^{N}$. For a subset $I=\{i_1< \cdots<i_k\}$ of $\{1, \dots, N\}$ let $(\R^{N})_I$ be the susbpace spanned by the vectors $e_{i_1}, \dots, e_{i_N}$. The monomials
\begin{align}
    e_I \coloneqq e_{i_1} \wedge \cdots \wedge e_{i_k}
\end{align}
form an orthonormal basis of $\wedge \R^{N}$. The Berezinian integral is the projection onto the top dimensional component $e_1 \wedge \dots \wedge e_N$
\begin{align}
    \int^B \colon \Omega^{i,\bullet} \longrightarrow \Omega^N(E), \qquad
    \alpha = \omega \otimes \nu \longmapsto \omega \otimes \langle \nu,e_1 \wedge \dots \wedge e_N\rangle.
\end{align}
Similarly, we have the projections
\begin{align}
    \Omega^{i,\bullet} \longrightarrow \Omega^i(E), \qquad
    \alpha = \omega \otimes \nu \longmapsto \alpha_I \coloneqq \omega \langle \nu,e_I \rangle.
\end{align}
For disjoint subsets $I_1, \dots,I_k \subset \{ 1,\dots, N\}$ let define $\epsilon(I_1, \dots,I_k)= \pm 1$ by
\begin{align}
    e_{I_1} \wedge \cdots \wedge e_{I_k}=\epsilon(I_1, \dots,I_k) e_{I_1 \cup \cdots \cup I_k}.
\end{align}
\begin{lem} \label{splitlem} Let $\alpha \in \Omega^{i,\bullet}$ and $\beta \in \Omega^{k,\bullet}$. We have
\begin{align}
 (\alpha \wedge \beta)_J = \sum_{I \subset J}  \epsilon(I,J \smallsetminus I) (-1)^{k \vert I \vert } \alpha_I \wedge \beta_{J-I}.
\end{align}
\end{lem}
\begin{proof}
    We can write 
    \begin{align}
\alpha = \sum_{I_1 \subset \{ 1,\dots, N\}} \alpha_{I_1} \otimes e_{I_1}, \quad \beta = \sum_{I_2 \subset \{ 1,\dots, N\}} \beta_{I_2} \otimes e_{I_2}. 
    \end{align}
    Hence, we have
    \begin{align}
        \alpha \wedge \beta & = \sum_{I_1,I_2 \subset \{ 1,\dots, N\}} (-1)^{k \vert I_1 \vert } \alpha_{I_1} 
 \wedge \beta_{I_2} \otimes e_{I_1} \wedge e_{I_2} \nonumber \\
 & = \sum_{I_1,I_2 \subset \{ 1,\dots, N\}} \epsilon(I_1,I_2)(-1)^{k \vert I_1 \vert } \alpha_{I_1} 
 \wedge \beta_{I_2} \otimes e_{I_1 \cup I_2}.
    \end{align}
    Taking the $J$-th component will kill all the terms except when $I_1$ and $I_2$ satisfy $I_1 \cup I_2 =J$.
\end{proof}

\subsubsection{Connection forms, covariant derivative and curvature form.} Let $\pi$ be the projection of $\g$ onto  $\kfrak=\so(\R^{N})$ defined by $\pi(X)=\frac{1}{2}(X-X^t)$. Note that $\pi \circ \Ad(k)=\Ad(k) \circ \pi$ for $k \in K$, where the adjoint map is $\Ad(k)X=kXk^{-1}$. The Maurer-Cartan form is the $1$-form
\begin{align}
   \vartheta\coloneqq g^{-1}dg \in \Omega^1(\GL_N(\R)^+)\otimes \g,
\end{align}
and we define the connection form
\begin{align}
    \theta \coloneqq \pi(\vartheta)=\frac{1}{2}(\vartheta-\vartheta^{t}) \in \Omega^1(\GL_N(\R)^+)\otimes \kfrak.
\end{align}
By pulling back by the projection map $\GL_N(\R)^+ \times \R^{N} \longrightarrow \GL_N(\R)^+$ we also get a connection form
\begin{align}
    \theta=\pi(\vartheta)=\frac{1}{2}(\vartheta-\vartheta^{t}) \in \Omega^1(\GL_N(\R)^+ \times \R^{N})\otimes \kfrak.
\end{align}
The connection form defines a covariant derivative
\begin{align}
    \nabla \colon \Omega^{i,j} \longrightarrow \Omega^{i+1,j} 
\end{align}
by $\nabla (\omega \otimes \nu)=(d \omega) \otimes \nu + (-1)^i \omega \otimes (\theta \nu)$ where we view 
\begin{align}
    \theta \in \Omega^1(\GL_N(\R)^+ \times \R^{N})\otimes \kfrak \subset \Omega^1(\GL_N(\R)^+ \times \R^{N})\otimes \so( \wedge \R^{N}).
\end{align}

The connection is compatible with the metric induced by the bilinear form, in the sense that for two sections $s_1,s_2 \in \Omega^{0,j}=\Omega^0(E,\wedge^j E_{\taut})$ we have
\begin{align}
    d \langle s_1,s_2\rangle = \langle \nabla s_1, s_2 \rangle + \langle s_1,\nabla s_2 \rangle.
\end{align}
This can be seen from the fact that $\theta^t=-\theta$, which implies $\langle \theta s_1, s_2 \rangle + \langle s_1,\theta s_2 \rangle =0.  
$

If we apply the covariant derivative twice we get $ \nabla^2 (\omega \otimes  s) = \omega \otimes R s$ where 
\begin{align}
    R \in [\Omega^2(\GL_N(\R)^+ \times \R^{N}) \otimes \kfrak]^K
\end{align}
is the curvature form. We can identify $\kfrak$ with $\wedge^2\R^{N}$ by the map
\begin{align} \label{idenso}
    \kfrak \longrightarrow \wedge^2\R^{N}, \qquad
    X  \longmapsto &  -\frac{1}{2}\sum_{i,j} X_{ij}e_i \wedge e_j.
\end{align}

\subsubsection{The Mathai-Quillen form.} \label{subsubsec : mqform} Let us denote by $v \colon E \longrightarrow E_{\taut}$ the tautological section defined $v[g_0,v_0]=[g_0,v_0,v_0]$. It defines an element $v \in \Omega^{0,1}$. Hence, we get a form
\begin{align}
-2\pi \lVert v \rVert^2-2\sqrt{\pi}\nabla v +R \in \Omega^{0,0} \oplus \Omega^{1,1} \oplus \Omega^{2,2}.    
\end{align}
The Mathai-Quillen form is the rapidly decreasing form defined by
\begin{align}
    U \coloneqq\delta_N\int^B \exp \left (-2\pi \lVert v \rVert^2-2\sqrt{\pi}\nabla v +R  \right ) \in \Omega^{N}_\rd(E)^{\GL_N(\R)^+},
\end{align}
where $\delta_N \coloneqq (-1)^{\frac{N(N-1)}{2}}(2\pi)^{-\frac{N}{2}}$.
\begin{prop}\cite[Theorem.~4.10]{mq} \label{Uprop} The form $U$ is closed, $\GL_N(\R)^+$-invariant and of integral one along the fibers.
\end{prop}

\subsection{The form $\varphi$} \label{sec:varphiform}
For $\vbf=(v,w)$ we define a section
\begin{align}
 s_\vbf \colon X  \longrightarrow E, \qquad z \longmapsto  s_\vbf(z) \coloneqq \left [g_z, \frac{g_z^{-1}v-g_z^tw}{\sqrt{2}} \right ]
\end{align}
where $g_z$ is such that $z=gg^t \in X$ (equivalently $z=g_z\SO(N)$ ). We define the form
\begin{align}
    \varphi^0(z,\vbf) \coloneqq s_\vbf^\ast U \in \Omega^N(X)\otimes C^\8(V_\R).
\end{align}

\subsubsection{Explicit formula.} Let $\lambda \coloneqq \frac{1}{2}(\vartheta+\vartheta^t) \in \Omega^1(\GL_N(\R)^+) \otimes \End(\R^{N})$ and let $\lambda_{ij} \in \Omega^1(\GL_N(\R)^+)$ be its $(i,j)$-entry, where we identify $\End(\R^{N}) \simeq \Mat_N(\R)$. For a subset $I=\{i_1 < \cdots <i_k\} \subset \{ 1, \dots, N \}$ and a function $\sigma \colon I \longrightarrow \{ 1, \dots, N \}$ we define the $\vert I \vert$-form
\begin{align}
    \lambda(\sigma) \coloneqq \lambda_{i_1\sigma(i_1)} \wedge  \cdots \wedge \lambda_{i_k\sigma(i_k)} \in \Omega^{\vert I \vert}(\GL_N(\R)^+)
\end{align}
and the generalized Hermite polynomial $H_\sigma \in \C[\R^{N}]$ by
\begin{align}
    H_\sigma(v)  \coloneqq \prod_{m=1}^NH_{d_m}(v_m)
\end{align}
where $d_m = \vert \sigma^{-1}(m) \vert$ and 
\begin{align}
H_d(t) \coloneqq \left (2t-\frac{d}{dt} \right )^d \cdot 1
\end{align} denotes the single variable Hermite polynomial.  This normalization of the Hermite polynomial is sometimes called the {\em physicist's Hermite polynomials} and the first few Hermite polynomials are: $H_1(t)=2t$, $H_2(t)=4t^2-2$, $H_3(t)=8t^3-12t$, ...

\begin{lem} \label{exphermite} Let $I \subset \{1,\dots,N\}$ be a subset. Then the $I$-th component is
\begin{align*}
\exp \left (-2\sqrt{\pi} \nabla s_\vbf+R \right )_I=(-1)^{\frac{\vert I \vert (\vert I \vert -1)}{2}}2^{-\frac{\vert I \vert }{2}}\sum_{\sigma}  H_\sigma\left (\sqrt{\pi}(g_z^{-1}v+g_z^{t}w) \right ) \otimes \lambda(\sigma),
\end{align*}
where the sum is over all functions $\sigma \colon I \longrightarrow \{1, \dots, N\}$.
\end{lem}
\begin{proof} For $I=\{1,\dots,N\}$ this is \cite[Lemma.~4.4]{rbrkmmq}.
Note that $d(g_z^{-1})g=-g_z^{-1}dg=-\vartheta$ and $d(g_z^{t})g_z^{-t}=\vartheta^t$. Hence
\begin{align}
    ds_\vbf=- \frac{\vartheta g_z^{-1}v+\vartheta^tg_z^{t}w}{\sqrt{2}}
\end{align}
and
\begin{align}
    \nabla s_\vbf = ds_\vbf+\theta s_\vbf = -\lambda\frac{g_z^{-1}v+g_z^tw}{\sqrt{2}}
\end{align}
where $\lambda=\frac{1}{2}(\vartheta+\vartheta^t)$. We write $(g_z^{-1}v+g_z^{t}w)=\sum_m(g_z^{-1}v+g_z^{t}w)_m\otimes e_m$ to get
\begin{align}
    \nabla s_\vbf= -\sum_{m=1}^N \frac{(g_z^{-1}v+g_z^{t}w)_m}{\sqrt{2}} \lambda_{m},
\end{align}
where
\begin{align}
    \lambda_m \coloneqq \sum_{l=1}^N \lambda_{lm} \otimes e_l \in \Omega^{1,1}
\end{align}
and $\lambda_{lm} = \frac{1}{2}(\vartheta_{lm}+\vartheta_{ml})$. The curvature is the $2$-form $ R \in [\Omega^2(\GL_N(\R)^+ \times \R^{N}) \otimes \kfrak]^K$ given by $R=d\theta+\theta^2$.
Note that
\begin{align}
    \theta^2 = \frac{1}{4}(\vartheta-\vartheta^t) (\vartheta-\vartheta^t) = \frac{1}{4}(\vartheta^2 -\vartheta \vartheta^t - \vartheta^t \vartheta + (\vartheta^t)^2)
\end{align} and
\begin{align}
    d \theta & = \frac{(d g_z^{-1})dg_z-(dg_z^t)(dg_z^{-t})}{2} = -\frac{\vartheta^2 + (\vartheta^t)^2}{2}.
\end{align} It follows that $R=- \lambda^2$. By the identification $\kfrak \simeq \wedge^2\R^{N}$ in \eqref{idenso}, we can see the curvature as an element $ R \in \Omega^{2,2}=[\Omega^2(\GL_N(\R)^+ \times \R^{N}) \otimes \wedge^2\R^{N})]^K$. We have
\begin{align}
 R =-\frac{1}{2}\sum_{k,l} R_{kl} e_k \wedge e_l  = - \frac{1}{2} \sum_{m=1}^N \lambda_m^2.
\end{align}

For $t\in \R$, we have
\begin{align}
\exp(\sqrt{2\pi}t\lambda_m-\lambda_m^2/2)=\sum_{d=0}^\8 \frac{2^{-d/2}}{d!}H_d(\sqrt{\pi}t)\lambda_m^d.
\end{align}
Following the computations after \cite[Lemma.~4.4]{rbrkmmq} we find that
\begin{align} \label{equationforexp}
 \exp \left (-2\sqrt{\pi} \nabla s_\vbf+R \right ) = \sum_{\substack{d_1, \dots,d_N \\ d_m \geq0}}\frac{2^{-(d_1+\cdots+d_N)/2}}{d_1!\cdots d_N!}\prod_{m=1}^NH_{d_m}(\sqrt{\pi}(g_z^{-1}v+g_z^tw)_m) \lambda_m^{d_m}.
\end{align}
The $I$-th component of $\lambda_1^{d_1} \wedge \cdots \wedge \lambda_N^{d_N}$ is only nonzero when $d_1+ \cdots + d_N=\vert I \vert$. In that case, we have
\begin{align}
(\lambda_1^{d_1} \wedge \cdots \wedge \lambda_N^{d_N})_I & = \left (\bigwedge_{m=1}^N \left ( \sum_{l=1}^N \lambda_{lm} \otimes e_l \right )^{d_m} \right )_I \\
& = \left ( \bigwedge_{m=1}^N d_m! \sum_{\substack{I_m \subset I \\ \vert I_m \vert =d_m}} (\lambda_{l_1m} \otimes e_{l_1}) \wedge \cdots \wedge (\lambda_{l_{d_m}m} \otimes e_{l_{d_m}})  \right )_I
\end{align}
where the sum is over subsets $I_m=\{l_1< \dots <l_{d_m}\} \subset I$ of size $d_m$ and the factorial comes from reordering the terms. Expanding the product shows that $(\lambda_1^{d_1} \wedge \cdots \wedge \lambda_N^{d_N})_I$ is equal to
\begin{align}
(-1)^{\frac{\vert I \vert(\vert I \vert-1)}{2}}\sum_{\substack{I_1, \dots,I_N \\ I=I_1 \cup \cdots \cup I_N}} d_1! \cdots d_N!  (\lambda_{I_1} \wedge \cdots \wedge \lambda_{I_N} ) \otimes ( e_{I_1} \wedge \cdots \wedge e_{I_N}).
\end{align}
To every partition of $I$ into subsets $I_1,\dots,I_N$ corresponds a unique function $\sigma \colon  I \longrightarrow \{1, \dots, N\}$ determined by $\sigma(i)=m$ if and only if $i \in I_{m}$. Note that $\vert \sigma^{-1}(m) \vert=\vert I_m \vert=d_m$. Recall that $\epsilon(I_1, \dots, I_N) = \pm 1$ is the sign that appears when we reorder
\begin{align}
e_{I_1} \wedge \cdots \wedge e_{I_N}=\epsilon(I_1, \dots, I_N) e_I.
\end{align}
The same sign appears when we reorder
\begin{align}
\lambda_{I_1} \wedge \cdots \wedge \lambda_{I_N}=\epsilon(I_1, \dots, I_N) \lambda(\sigma).
\end{align}
Hence, we conclude that
\begin{align}
(\lambda_1^{d_1} \wedge \cdots \wedge \lambda_N^{d_N})_I= (-1)^{\frac{\vert I \vert(\vert I \vert-1)}{2}}\sum_{\substack{\sigma \\ \vert \sigma^{-1}(m) \vert =d_m}} d_1! \cdots d_N! \lambda(\sigma),
\end{align}
where the sum is over all functions $\sigma$ such that $\vert \sigma^{-1}(m) \vert =d_m$ for all $m$. Since we sum over all nonnegative $d_m$ such that $d_1+ \cdots + d_N=\vert I \vert$, we have
\begin{align}
\sum_{\substack{d_1, \dots,d_N \\ d_m \geq0}}\sum_{\substack{\sigma \\ \vert \sigma^{-1}(m) \vert =d_m}}=\sum_{\sigma }.
\end{align}
\end{proof}
After multiplying with the exponential term and the constant $\delta_N$ we get the following closed formula for $\varphi^0$.
\begin{prop} \label{explicitphihat}
    The form is explicitly
    \begin{align}
    \varphi^0(z,\vbf)= 2^{-N}\pi^{-\frac{N}{2}}\sum_{\sigma}\varphi^0_\sigma(\rho_{g_z}^{-1}\vbf) \otimes \lambda(\sigma) \in \Omega^N(X) \otimes C^\8(V_\R)
\end{align}
where the sum is over functions $\sigma \colon \{ 1, \dots, N \} \longrightarrow \{ 1, \dots, N \}$
and
\begin{align}
    \varphi^0_\sigma(\vbf)\coloneqq H_\sigma\left (\sqrt{\pi}(v+w) \right ) e^{-\pi  \lVert v- w \rVert^2}.
\end{align}
\end{prop}

The form $\varphi^0$ is not rapidly decreasing in the variable $\vbf \in V_\R$. To obtain a rapidly decreasing form, we define 
\begin{align}
\varphi \in \Omega^N(X) \otimes \Scal(V_\R)
\end{align} by  setting $\varphi(z,\vbf) \coloneqq e^{-2\pi Q(\vbf)}\varphi^0(z,\vbf).$
We have
\begin{align} 
\varphi(z,\vbf) 
& = 2^{-N}\pi^{-\frac{N}{2}}\sum_{\sigma}\varphi_\sigma(\rho_{g_z}^{-1}\vbf) \otimes \lambda(\sigma)
\end{align}
where
\begin{align}
    \varphi_\sigma(v,w) \coloneqq  \varphi^0_\sigma(v,w)e^{-2\pi Q(\vbf)} = H_\sigma\left (\sqrt{\pi}(v+w) \right ) e^{-\pi  \lVert v \rVert^2- \pi \lVert w \rVert^2}.
\end{align}
\subsubsection{Relation to the Kudla-Millson form} \label{relation kudla millson}
We refer to \cite{rbrkmmq} for more details on the construction of the Kudla-Millson form via the Mathai-Quillen formalism. Let $\D \simeq \SO(N,N)^+/\SO(N)\times \SO(N)$ be the Grassmannian of negative planes in the quadratic space $V_\R$ with the quadratic form $Q(\vbf)$ of signature $(N,N)$. The representation $\rho$ induces an embedding
\begin{align}
\rho \colon X \hooklongrightarrow \D, \quad g \longmapsto \begin{pmatrix}
g & 0 \\ 0& g^{-t}
\end{pmatrix}.
\end{align}
The Kudla-Millson form is an $N$-form
\begin{align}
\varphi_{\KM} \in \Omega^{N}(\D) \otimes \Scal(V_\R).
\end{align}
Let $E'$ be the tautological rank $N$ bundle $E'=\SO(N,N)^+\times_{\SO(N)^2}\R^N$ over $\D$, and let $U' \in \Omega^N(E')$ be the Thom form. By \cite{rbrkmmq}, there is a section $\widetilde{s}_\vbf \colon \D \longrightarrow E'$ such that $\varphi_{\KM}(\vbf)=\widetilde{s}_\vbf^\ast U'$. Moreover, the pullback of $E'$ by $\rho$ is precisely $E$ and the following diagram is commutative:
\begin{equation}
\begin{tikzcd}[column sep=large,row sep=large]
E \arrow[r] & E'  \\
X \arrow[u, "s_\vbf"] \arrow[r, swap, "\rho"] & \D \arrow[u, swap, "\widetilde{s}_\vbf"].
\end{tikzcd}
\end{equation}
Thus, we have
\begin{align}
\varphi=\rho^\ast \varphi_{\KM}.
\end{align}

\subsection{The form $\psi$}\label{sec:psiform}

We have identified $X \simeq S \times \R_{>0}$ as a bundle over $S$ by the map
\begin{align}
    X \longrightarrow S \times \R_{>0}, \qquad z \longmapsto \left ( p(z), u\right )
\end{align}
where $p(z) \coloneqq \det(z)^{-\frac{1}{N}} z$ and $u\coloneqq \det(z)^{\frac{1}{N}}$. A form $\eta(z) \in \Omega^N(X)$ can be written
\begin{align}
\eta(z)=\etatil(z) + (-1)^{N-1} \left ( \iotau\eta(z) \right )\frac{du}{u}
\end{align}
where $\etatil(z) \in \Omega^N(X)$ and  $\iotau \eta(z) \in \Omega^{N-1}(X)$ are of degree $0$ along $\R_{>0}$.
 
As in section \ref{subsec: thom iso}, the integration along the fibers $\R_{>0}$ of the fiber bundle $p \colon X \longrightarrow S$ induces the pushforward
\begin{align}
  p_\ast \colon \Omega_{\rd}^N(X) \longrightarrow \Omega^{N-1}(S), \qquad
     \eta \longmapsto p_\ast \eta \coloneqq (-1)^{N-1}\int_0^\8 \iotau \eta \frac{du}{u},
\end{align}
where $\Omega_{\rd}^N(X)$ is the space of forms that are rapidly decreasing as $u$ goes to $0$ and $\8$. Note that the bundle can be seen as a vector bundle after identifying $\R_{>0} \simeq \R$ via the exponential map. The pushforward is then as in \cite[p.~61]{botu}. 
Let $\psi^0(z,\vbf)$ be the form
\begin{align} 
\psi^0(z,\vbf) \coloneqq (-1)^{N-1}\iota_{u\frac{\partial}{\partial u}}\varphi^0(z,\vbf)  \in \Omega^{N-1}(X) \otimes C^\8(V_\R)
\end{align}
and $\psi(z,\vbf) \coloneqq e^{-2\pi Q(\vbf)}\psi^0(z,\vbf)$ the contraction of $\varphi$. The form $\varphi(z,\vbf)$ splits as
\begin{align}   \varphi(z,\vbf)= \vptil(z,\vbf)+\psi(z,\vbf) \frac{du}{u}.
\end{align}
Moreover, since the form $\varphi$ is closed, by taking the derivative on both sides we find the transgression formula
\begin{align} \label{transgression1}
    d_1\psi(z,\vbf)= u\frac{\partial}{\partial u} \vptil(z,\vbf) \in \Omega^{N}(X)\otimes\Scal(V_\R)
\end{align}
where $d_1$ is the derivative on $S$.

\subsubsection{Explicit formula of $\psi^0$.} To compute this form, note that we can write the connection $\vartheta=g^{-1}dg$ in the coordinates $(g_1,u)$ of $\SL_N(\R) \times \R_{>0}$ where $g=ug_1$. We have
\begin{align}
    g^{-1}dg^{-1}=u^{-1}g_1^{-1}d(g_1u)=\vartheta_1+\frac{du}{u}
\end{align}
where $\vartheta_1=g_1^{-1}dg_1$. If we set $\lambda^{(1)} \coloneqq \frac{1}{2} \left ( \vartheta_1+\vartheta^t_1 \right ) \in \Omega^1(\SL_N(\R)) \otimes \End(\R^{N})$ then
\begin{align} \label{lambda1u}
    \lambda=\lambda^{(1)}+\frac{du}{u}.
\end{align}
By taking the product we can then write
\begin{align}
\lambda(\sigma)=\lambda^{(1)}(\sigma)+(-1)^{N-1}\left (\iota_{u\frac{\partial}{\partial u}}\lambda(\sigma) \right )\frac{du}{u}
\end{align}
where $\iota_{u\frac{\partial}{\partial u}}\lambda(\sigma) \in \Omega^{N-1}(S)$.

\begin{prop} \label{explicitpsi}
    The form  $\psi^0(z,\vbf)  \in \Omega^{N-1}(X)\otimes\Scal(V_\R)$ is equal to
\begin{align}
 \psi^0(z,\vbf) & =(-1)^{N-1} 2^{-N}\pi^{-\frac{N}{2}} \sum_{\sigma}\varphi^0_\sigma(\rho_{g_z}^{-1}\vbf) \otimes \left (\iota_{u\frac{\partial}{\partial u}}\lambda(\sigma) \right ).
\end{align}
\end{prop}

From the proof of Lemma \ref{exphermite}, we have
\begin{align}
\varphi(z,\vbf) & =\delta_Ne^{-\pi \lVert g_z^{-1}v\rVert^2-\lVert g_z^{t}w \rVert^2 }\int^B\exp \left (\sqrt{2\pi} \lambda(g_z^{-1}v+g_z^tw)+R \right ).
\end{align}
By writing $\lambda=\lambda^{(1)}+\frac{du}{u}$ as in \eqref{lambda1u}, the Berezinian is equal to
\begin{align}
\int^B  \exp \left (\sqrt{2\pi} \frac{du}{u}(g_z^{-1}v+g_z^tw) \right ) \exp \left (\sqrt{2\pi} \lambda^{(1)}(g_z^{-1}v+g_z^tw)+R \right ) .
\end{align}
The first term is equal to 
\begin{align}
\exp \left (\sqrt{2\pi} \frac{du}{u}(g_z^{-1}v+g_z^tw) \right )=1+ \sqrt{2\pi}\sum_{l=1}^N (g_z^{-1}v+g_z^tw)_l \left (\frac{du}{u} \otimes e_l \right ).
\end{align}
Thus, the $\frac{du}{u}$-component of the Berezinian is
\begin{align}
\sqrt{2\pi}\sum_{l=1}^N (g_z^{-1}v+g_z^tw)_l \int^B\left (\frac{du}{u} \otimes e_l \right ) \wedge \exp \left (\sqrt{2\pi} \lambda^{(1)}(g_z^{-1}v+g_z^tw)+R \right ) .
\end{align}
Using\footnote{Note that here $\epsilon(\{l\},\{1,\dots,N\}\smallsetminus\{l\})=(-1)^{l-1}$ and $(-1)^{k \vert \{l\} \vert } =(-1)^{N-1}$.}  Lemma \ref{splitlem} for $\alpha=\frac{du}{u} \otimes e_l$ and $\beta=\exp \left (\sqrt{2\pi} \lambda^{(1)}(g_z^{-1}v+g_z^tw)+R \right )$ we find that the Berezinian in the last equation is
\begin{align}
\sum_{l=1}^N (-1)^{l-1}  \exp \left (\sqrt{2\pi} \lambda^{(1)}(g_z^{-1}v+g_z^tw)+R \right )_{I_{\widehat{l}}} \frac{du}{u}.
\end{align}
where $I_{\widehat{l}} \coloneqq I_{\{1,\dots,N\}\smallsetminus\{l\}}$. We deduce the following.

\begin{prop} \label{explicitpsi2}
The form $\psi^0(z,\vbf)$ is also equal to
\begin{align}
(-1)^{N-1}\delta_{N-1}e^{-\pi \lVert g_z^{-1}v\rVert^2-\lVert g_z^{t}w \rVert^2 } \sum_{l=1}^N (g_z^{-1}v+g_z^tw)_l \exp \left (\sqrt{2\pi} \lambda^{(1)}(g_z^{-1}v+g_z^tw)+R \right )_{I_{\widehat{l}}}.
\end{align}
In particular, we have $\psi(z,0)=0$.
\end{prop}

\subsection{The transgression form} \label{transgression}
We have seen that the form $\psi$ already satisfies the transgression formula \eqref{transgression1} with respect to $\R_{>0}$ acting on $\vbf$ by $\rho$. We will need a transgression form with respect to the action of $\R_{>0}$ given by rescaling the vector. We define
\begin{align}
\Omega_t \coloneqq -2\pi t \lVert v \rVert^2-2\sqrt{\pi}\sqrt{t}\nabla v +R
\end{align}
for $t>0$, and the Thom form
\begin{align}
U_t = \delta_N \int^B\exp(\Omega_t) \in \Omega^N_{\rd}(E)^{\GL_N(\R)^+}
\end{align}
obtained by rescaling by $\sqrt{t}$ in the fiber of $E$. Note that $U_1=U$.
Let $\Util_t$ be the form defined by
\begin{align}
\Util_t\coloneqq -(-1)^{\frac{N(N-1)}{2}}2^{-\frac{N}{2}}\pi^{-\frac{N-1}{2}} \int^B v \wedge \exp(\Omega_t) \in  \Omega^{N-1}_{\rd}(E)^{\GL_N(\R)^+}
\end{align}
where $v \colon E \rightarrow E_{\taut}$ is the tautological section as in section \ref{subsubsec : mqform}.
\begin{prop} \label{prop transgression}
We have 
\begin{align}
\frac{\partial U_t}{\partial t}=\frac{1}{\sqrt{t}}d\Util_t.
\end{align}
In particular, the Thom form $U$ is exact outside of the zero section.
\end{prop}
\begin{proof} This proposition is \cite[Proposition.~1.53]{berline} with a slightly different normalization. We have
\begin{align}
\frac{\partial}{\partial t}\exp(\Omega_t)=\left (\frac{\partial \Omega_t}{\partial t} \right )\exp(\Omega_t)
\end{align}
and
\begin{align}
\frac{\partial \Omega_t}{\partial t}= -\frac{\sqrt{\pi}}{\sqrt{t}} \left ( \nabla+2\sqrt{\pi}\sqrt{t}\iota(v) \right ) v.
\end{align}
We have $\left ( \nabla+2\sqrt{\pi}\sqrt{t}\iota(v) \right )\exp(\Omega_t)=0$, hence
\begin{align}
d\int^B v \wedge  \exp(\Omega_t)& = \int^B \left ( \nabla+2\sqrt{\pi}\sqrt{t}\iota(v) \right ) \left ( v \wedge \exp(\Omega_t)\right ) \nonumber \\
& = -\frac{\sqrt{t}}{\sqrt{\pi}}\int^B \frac{\partial}{\partial t}\exp(\Omega_t) \nonumber \\
& = -\frac{\sqrt{t}}{\sqrt{\pi}}\frac{\partial}{\partial t} \int^B\exp(\Omega_t).
\end{align}
The result follows after multiplying by $\delta_N$ on both sides. 

Finally, the Thom form is exact on $E \smallsetminus E_0$ since
\begin{align}
U= d \int_1^\8 \Util_t \frac{dt}{\sqrt{t}},
\end{align}
where the integral converges by the rapid decrease (away from the zero section where $v=0$). Note that the right-handside is essentially the form $\tau$ in \cite[p.~106]{mq}. 
\end{proof}
Let $s_\vbf \colon X \longrightarrow E$ be the section considered previously. Let $\Util \coloneqq \Util_1$ and define $\alpha^0(z,\vbf) \coloneqq s_\vbf^\ast \Util \in \Omega^{N-1}(X) \otimes C^\8(V_\R)$. We have
\begin{align}
\alpha^0(z,\vbf) = -(-1)^{\frac{N(N-1)}{2}}2^{-\frac{N}{2}}\pi^{-\frac{N-1}{2}} \int^B s_\vbf(z) \wedge \exp(s_\vbf^\ast \Omega_1).
\end{align}
\begin{prop} For any real number $t>0$ we have
    \begin{align} \label{alphatilprop}
        d\alpha^0(z,\sqrt{t}\vbf)= t\frac{\partial}{\partial t}\varphi^0(z,\sqrt{t}\vbf).
    \end{align}
\end{prop}
\begin{proof} 
Since $s_\vbf^\ast \Omega_t=s_{\sqrt{t}\vbf}^\ast \Omega_1$,  we have $s_\vbf^\ast U_t=s_{\sqrt{t}\vbf}^\ast U=\varphi^0(z,\sqrt{t}\vbf)$. On the other hand, we have
\begin{align}
s_\vbf^\ast \Util_t & =  -(-1)^{\frac{N(N-1)}{2}}2^{-\frac{N}{2}}\pi^{-\frac{N-1}{2}} \int^B s_\vbf(z)\wedge \exp(s_\vbf^\ast \Omega_t) \nonumber \\
& = \frac{1}{\sqrt{t}} s_{\sqrt{t}\vbf}^\ast \Util \nonumber \\
& = \frac{1}{\sqrt{t}} \alpha^0(z,\sqrt{t}\vbf).
\end{align}
Hence, we have
\begin{align}
d\alpha^0(z,\sqrt{t}\vbf)=\sqrt{t}d(s_\vbf^\ast \Util_t)=t\frac{\partial}{\partial t}s_\vbf^\ast U_t=t\frac{\partial}{\partial t} \varphi^0(z,\sqrt{t}\vbf).
\end{align}
\end{proof}

\subsubsection{Explicit formula for $\alpha^0$.}  For $l \in \{1, \dots,N\}$ and a function 
\begin{align}
\sigma_l \colon I_{\widehat{l}} \longrightarrow  \{1, \dots,N\}
\end{align} 
we define a smooth function on $V_\R$
\begin{align}
\varphi^0_{\sigma_l}(\vbf) & \coloneqq (v_l-w_l) H_{\sigma_l}(\sqrt{\pi}(v+w)) e^{-\pi  \lVert v-w \rVert^2 } \nonumber \\
& = (v_l-w_l)\prod_{m=1}^N H_{\vert \sigma_l^{-1}(m) \vert}\left (\sqrt{\pi}(v_m+w_m) \right ) e^{-\pi  \lvert v_m- w_m \rvert^2}.
\end{align}
\begin{prop} \label{explicitalphatil}
    The form $\alpha^0(z,\vbf)\in \Omega^{N-1}(X) \otimes C^\8(V_\R)$ is equal to
    \begin{align}
    \alpha^0(z,\vbf)=-2^{-N}\pi^{-\frac{N-1}{2}} \sum_{l=1}^N \sum_{\sigma_l}(-1)^{l-1} \varphi^0_{\sigma_l}(\rho_{g_z}^{-1}\vbf) \otimes \lambda(\sigma_l),
\end{align}
where the inner sum is over functions $\sigma_l \colon I_{\widehat{l}} \longrightarrow  \{1, \dots,N\}$.
\end{prop}
\begin{proof} We have
\begin{align}
  \int^B s_\vbf(z) \wedge \exp(s_\vbf^\ast \Omega_1) & = \frac{1}{\sqrt{2}} \sum_{l=1}^N  \int^B  ((g_z^{-1}v-g_z^tw)_l \otimes e_l) \wedge \exp \left (-2\sqrt{\pi} \nabla s_\vbf+R \right ).
\end{align}
Using\footnote{Note that here $\epsilon(\{l\},\{1,\dots,N\}\smallsetminus\{l\})=(-1)^{l-1}$ and $(-1)^{k \vert \{l\} \vert } =(-1)^{N-1}$.}  Lemma \ref{splitlem} for $\alpha=s_\vbf(z)$ and $\beta=\exp(s_\vbf^\ast \Omega_1)$ we find that
\begin{align}
\int^B s_\vbf(z) \wedge \exp(s_\vbf^\ast \Omega_1) & = \frac{(-1)^{N-1}}{\sqrt{2}}\sum_{l=1}^N (-1)^{l-1} (g_z^{-1}v-g_z^tw)_l \exp(s_\vbf^\ast \Omega_1)_{I_{\widehat{l}}}.
\end{align}
By Lemma \ref{exphermite}, for $I =I_{\widehat{l}}$ we have
\begin{align}
\exp \left (-2\sqrt{\pi} \nabla s_\vbf+R \right )_I=(-1)^{\frac{(N-1) (N-2)}{2}}2^{-\frac{N-1}{2}}\sum_{\sigma_l}  H_{\sigma_l}\left (\sqrt{\pi}(g_z^{-1}v+g_z^{t}w) \right ) \otimes \lambda(\sigma_l),
\end{align}
where the sum is over all functions $\sigma_l \colon I \longrightarrow \{1, \dots, N\}$. Combining the two gives
\begin{align}
\int^B s_\vbf(z) \wedge \exp(s_\vbf^\ast \Omega_1) & = (-1)^{\frac{N(N-1)}{2}}2^{-\frac{N}{2}}\sum_{l=1}^N \sum_{\sigma_l}  (-1)^{l-1} \varphi^0_{\sigma_l}(\rho_{g_z}^{-1}\vbf) \otimes \lambda(\sigma_l).
\end{align}
The result follows by multipliying by the constant $(-1)^{\frac{N(N-1)}{2}}2^{-\frac{N}{2}}\pi^{-\frac{N-1}{2}}$ and the exponential term $\exp\left (-\pi  \lVert v-w \rVert^2 \right )$.
\end{proof}

\subsection{Estimates}

\begin{lem}\label{lemmapolynomial} Let $P(\vbf) \in \C[V_\R]$ be a polynomial of degree $d$ and $\vbf=(v,w) \in \Lcal^\vee$ a regular vector. There is a constant $C>0$ depending on $P$ and $\Lcal$ such that
\begin{align}
\int_0^\8   P(u^{-1}v,uw) e^{-\pi \lVert u^{-1}v-uw \rVert^2} \frac{du}{u} \leq  C \lVert v \rVert^{\frac{d}{2}} \lVert w \rVert^{\frac{d}{2}} P(v,w) e^{-2\pi \left ( \lVert v \rVert \lVert w \rVert-Q(\vbf) \right ) }.
\end{align}
\end{lem}
\begin{proof}
For $I=\{i_1, \dots,i_N\} \in (\Z_{\geq 0})^N$ we write $v^I \coloneqq v_1^{i_1} \cdots v_N^{i_N}$. We can write $P$ as a sum of monomials
\begin{align}
P(v,w)=\sum_{I,J} a_{I,J}v^Iw^J,
\end{align} 
where the sum is over all $N$-tuples $I,J \in (\Z_{\geq 0})^N$ with $\vert I \vert+\vert J \vert \leq d$.
We have
\begin{align}
& \int_0^\8   P(u^{-1}v,uw) e^{-\pi y \lVert u^{-1}v-uw \rVert^2} \frac{du}{u} \\
&\qquad = \sum_{I,J} a_{I,J}v^Iw^J e^{2 \pi Q(\vbf)}\int_0^\8  e^{-\pi u^{-2}\lVert v \rVert^2-\pi u^2 \lVert w \rVert^2} u^{\vert J \vert-\vert I \vert}\frac{du}{u} \\
&\qquad = \frac{1}{2}\sum_{I,J} a_{I,J}v^Iw^J  \left ( \frac{\lVert v \rVert}{\lVert w \rVert} \right )^{\frac{\vert J \vert-\vert I \vert}{2}} e^{2 \pi Q(\vbf)} K_{\frac{\vert J \vert-\vert I \vert}{2}}(\pi \Vert v \rVert \lVert w \rVert).
\end{align}
For a fixed real number  $r$ there is a positive constant $C_r$ (depending on $r$) such that for all $a \gg 0$
\begin{align}
K_{r}(a) \leq C_re^{-2a};
\end{align}
see \cite[K7 p.~271]{langelliptic}. Since $\vbf \in \Lcal^\vee \subset \frac{1}{n}V_\Z$ for some integer $n \geq 1$, we have
\begin{align}
\left ( \frac{\lVert v \rVert}{\lVert w \rVert} \right )^{\frac{\vert J \vert-\vert I \vert}{2}} \leq \frac{\lVert v \rVert^\frac{\vert J \vert}{2} \lVert w \rVert^\frac{\vert I \vert}{2}}{\lVert v \rVert^\frac{\vert I \vert}{2} \lVert w \rVert^\frac{\vert J \vert}{2}}\leq n^d \lVert v \rVert^{\frac{d}{2}} \lVert w \rVert^{\frac{d}{2}}.
\end{align}  
\end{proof}

\subsection{Thom form properties}

Since the section satisfies $s_\vbf(gz)=g s_{\rho_g^{-1}\vbf}(z)$ for any $g \in \GL_N(\R)^+$, it is $\Gamma_\vbf$-equivariant and descends to a section
\begin{align}
    s_\vbf \colon \Gamma_\vbf \backslash X & \longrightarrow \Gamma_\vbf \backslash E.
\end{align}
Hence, for fixed $\vbf$ we have $\varphi^0(z,\vbf) \in \Omega^N(X)^{\Gamma_{\vbf}}$.

\begin{prop} \cite[Proposition.~6.24(b)]{botu} \label{thomformprop} Let $E_0$ be the image of the zero section and let $\omega \in \Omega_c^{(N^2-N)/2}(\Gamma_\vbf \backslash E)$ be a compactly supported form. Then
\begin{align} \label{integralU}
\int_{\Gamma_\vbf \backslash E} \omega \wedge U=\int_{\Gamma_\vbf \backslash E_0} \omega.
\end{align}
\end{prop}


\begin{prop} \label{thomformprop2} Let $\vbf$ be a regular vector. 
\begin{enumerate} \item  If $Q(\vbf)>0$, then the form $\varphi^0(z,\vbf)$ is a Poincaré dual to $\Gamma_\vbf \backslash X_\vbf \subset \Gamma_\vbf \backslash X$. Thus, if $\omega \in \Omega_c^{(N^2-N)/2}(\Gamma_\vbf \backslash X)$ is a compactly supported form, then
\begin{align} 
\int_{\Gamma_\vbf \backslash X} \omega \wedge \varphi^0(z,\vbf)=\int_{\Gamma_\vbf \backslash X_\vbf} \omega.
\end{align}
\item  If $Q(\vbf) \leq 0$, then the form $\varphi^0(z,\vbf)$ is exact.
\end{enumerate}
\end{prop}

\begin{proof} The map $[g,v]  \longrightarrow (z=gg^t,gv)$ identifies the vector bundle $E$ with  $X \times \R^N$. After composing with this map, the section $s_\vbf$ becomes $s_\vbf(z)=(z,v-zw)$. This map has rank $N$ since $w \neq 0$.  Let $E_\vbf \subset E$ be the image of $s_\vbf$. If $Q(\vbf)>0$, then $E_\vbf$ intersects the zero section $E_0$ transversally. The pullback of a Poincaré dual of $\Gamma_\vbf \backslash E_0$ is then a Poincaré dual of $\Gamma_\vbf \backslash X_\vbf$. This proves the first part.

When $Q(\vbf) \leq 0$, then $X_\vbf$ is empty and $E_\vbf$ is contained in $E \smallsetminus E_0$. Since $U$ is exact on $E \smallsetminus E_0$ by Proposition \ref{prop transgression}, it follows that $\varphi^0(z,\vbf)$ is exact. More precisely, the pullback of $U= d \int_1^\8 \Util_t \frac{dt}{\sqrt{t}}$ in Proposition \ref{prop transgression} gives
\begin{align}
\varphi^0(z,\vbf)=d\int_1^\8 \alpha^{0}(z,\sqrt{t}\vbf) \frac{dt}{t}.
\end{align}

\end{proof}
%

\section{Theta lift}

\subsection{Weil representation} \label{weil theta}
We consider the Weil representation of $\GL_N(\R) \times \SL_2(\R)$ on the space $\Scal(V_\R)$ of Schwartz functions in two different models. Let $e \colon \R \longrightarrow U(1)$ be the character $e(t)\coloneqq e^{2i\pi t}$.

In the first model the representation $\omega \colon \GL_N(\R) \times \SL_2(\R) \longrightarrow U(\Scal(V_\R))$ is defined by the formulas
\begin{align*}
\omega\left (g, h \right ) \varphi(\vbf)&=\omega(1,h)\varphi(\rho_g^{-1}\vbf), \\
    \omega\left (1, \begin{pmatrix}
        a & 0 \\ 0 & a^{-1}
    \end{pmatrix} \right ) \varphi(\vbf)&=\vert a \vert^N \varphi(a\vbf), \qquad a \in \R^\times, \\
    \omega\left (1, \begin{pmatrix}
        1 & n \\ 0 & 1
    \end{pmatrix} \right ) \varphi(\vbf)&=e \left ( nQ(\vbf) \right ) \varphi(\vbf), \qquad n \in \R, \\
    \omega\left (1, \begin{pmatrix}
        0 & -1 \\ 1 & 0
    \end{pmatrix} \right ) \varphi(\vbf)&= \int_{V_\R} \varphi(\vbf') e \left (- B(\vbf,\vbf') \right )d\vbf',
\end{align*}
where $\vbf'=(v',w')$ and $\vbf=(v,w)$. In this orthogonal-symplectic model, the bilinear form is $B(\vbf,\vbf')=\langle v',w \rangle+ \langle v,w' \rangle$, with associated quadratic form $Q(\vbf)=\frac{1}{2}B(\vbf,\vbf)$.   

A second model $\omega' \colon \GL_N(\R) \times \SL_2(\R) \longrightarrow U(\Scal(V_\R))$ for the Weil representation is given by
by
\begin{align}
    \omega'(g,h)\varphi(\vbf)\coloneqq \vert \det(g)\vert^{-1}\varphi(g^{-1}\vbf h)
\end{align}
where $h$ acts on the columns of $\vbf$. The two representations are intertwined by a partial Fourier transform
\begin{align}
    \four_2 \colon \Scal(V_\R) \longrightarrow \Scal(V_\R),
\end{align}
defined by 
\begin{align}
   \four_2(\varphi)(v,w) \coloneqq \int_{\R^{N}} \varphi(v,w') e(\langle w,w' \rangle)dw'.
\end{align}
Similarly, we define the Fourier transform in the first variable
\begin{align}
    \four_1(\varphi)(v,w) \coloneqq \int_{\R^{N}} \varphi(v',w) e(\langle v,v' \rangle)dv'.
\end{align}
One can easily verify the following Lemma.
\begin{lem} \label{lem intertwiner}
    We have 
    \begin{align}
        \four_2(\omega(g,h)\varphi)(\vbf)=\omega'(g,h)\four_2(\varphi)(\vbf)
    \end{align}
and
\begin{align}
        \four_1(\omega(g,h)\varphi)(\vbf)=\omega'(g^\ast,h^\ast)\four_1(\varphi)(\vbf),
\end{align}
where $g^\ast=g^{-t}$ and $h^\ast$ is defined by 
\begin{align}
 \begin{pmatrix} a & b \\ c & d \end{pmatrix}^\ast=\begin{pmatrix} 0 & 1 \\ 1 & 0 \end{pmatrix} \begin{pmatrix} a & b \\ c & d \end{pmatrix}\begin{pmatrix} 0 & 1 \\ 1 & 0 \end{pmatrix}=\begin{pmatrix} d & c \\ b & a \end{pmatrix}.
 \end{align}
\end{lem}
%

Let $\Lcal=L \times L$ as previously, and let $\chi\in \C[D_\Lcal]$ be a test function. Note that $D_\Lcal \simeq  D_L \times D_L$ where $D_L\coloneqq L^\ast/L$. We can also define a finite partial Fourier transform 
\begin{align}
\four_2 \colon \C[D_\Lcal] \longrightarrow \C[D_\Lcal],
\end{align}
by taking a finite Fourier transform in the second variable
\begin{align}
\four_2(\chi)(v_0,w_0') \coloneqq \frac{1}{\sqrt{\vert D_L \vert }} \sum_{w_0 \in D_L} \chi(v_0,w_0)e(\langle w_0,w_0' \rangle).
\end{align}
We define $\four_1 \colon \C[D_\Lcal] \longrightarrow \C[D_\Lcal]$ analogously. 

\begin{prop}[Poisson summation] We have
    \begin{align}
\sum_{\vbf \in  V} \omega(g,h)\varphi(\vbf)\chi(\vbf) & =\sum_{\vbf \in  V} \omega'(g,h)\four_2(\varphi)(\vbf)\four_2(\chi)(\vbf) \\
& =\sum_{\vbf \in  V} \omega'(g^\ast,h^\ast)\four_1(\varphi)(\vbf)\four_1(\chi)(\vbf).
\end{align}
\end{prop}

\subsection{Computation of the partial Fourier transforms} 
For a vector $\mu \in \C^N$ and a function $\sigma$ partitioning $\{1,\dots,N\}$ let us write
\begin{align}
\mu_\sigma \coloneqq \mu_1^{d_1} \cdots \mu_N^{d_N}   
\end{align}
where $d_m=\vert\sigma^{-1}(m)\vert$.
\begin{prop} \label{phisigma}
    The partial Fourier transforms of $\varphi_\sigma$ and $\varphi_{\sigma_l}$ are
\begin{align}
    \four_2(\varphi_\sigma)(\vbf)& =i^N 2^N \pi^\frac{N}{2} (\overline{vi+w})_\sigma e^{-\pi  \lVert v\rVert^2-\pi \lVert w \rVert^2}
 \end{align}
 and
 \begin{align}
   \four_2( \varphi_{\sigma_l})(\vbf)& =-i^{N}2^{N-1} \pi^{\frac{N-1}{2}} \left [ \frac{2\pi \vert v_li+w_l \vert ^2+\vert \sigma_l^{-1}(l) \vert}{2\pi\overline{(v_li+w_l)}}\right ] \overline{(vi+w)}_{\sigma_l}e^{-\pi  \lVert v\rVert^2-\pi \lVert w \rVert^2}.
\end{align}
In particular, we have $\four_1(\varphi)(0)=\four_2(\varphi)(0)=\four_1(\alpha)(0)=\four_2(\alpha)(0)=0$.
\end{prop}
\begin{proof}  
The proposition follows from the computation of the partial Fourier transform of each of terms of
    \begin{align} \label{prod1}  \varphi_\sigma(v,w) 
    & =\prod_{m=1}^NH_{d_m}\left (\sqrt{\pi}(v_m+w_m) \right )\exp\left (-\pi v_m^2-\pi w_m^2 \right ).
\end{align}

A direct computation shows that the partial Fourier transform of $H_d(\sqrt{\pi}(x+y))e^{-\pi x^2-\pi y^2}$ is
\begin{align}
2^d \sqrt{\pi}^d(x+iy')^de^{-\pi x^2-\pi(y')^2}=i^d2^d \sqrt{\pi}^d\overline{(ix+y')}^de^{-\pi x^2-\pi(y')^2}.
\end{align} 

Recall that 
\begin{align}
\varphi_{\sigma_l}(\vbf) 
& = (v_l-w_l) \prod_{m=1}^N H_{\vert \sigma_l^{-1}(m) \vert}(\sqrt{\pi}(v_m+w_m))\exp\left (-\pi v_m^2-\pi w_m^2 \right ) 
).
\end{align} Here again, the partial Fourier transform is computed separately for each factor of the product. For  the factors $m\neq l$, we use the previous computations for $\varphi_\sigma$. For the factor $m=l$, the partial Fourier transform of $(x-y)H_d(\sqrt{\pi}(x+y))e^{-\pi x^2-\pi y^2}$ is
\begin{align} \label{fourier transform equality}
-i^{d+1}2^d \sqrt{\pi}^d \left [ \frac{2\pi \vert ix+y' \vert ^2+d}{2\pi\overline{(ix+y')}}\right ]\overline{(ix+y')}^de^{-\pi x^2-\pi(y')^2}.
\end{align}

 For the last statement, note that 
 \begin{align}
\varphi_\sigma(v,w)=\varphi_\sigma(w,v), \qquad \varphi_{\sigma_l}(v,w)=-\varphi_{\sigma_l}(w,v)
\end{align} We deduce that 
\begin{align}
\four_1(\varphi_\sigma)(v,w) =\four_2(\varphi_\sigma)(w,v), \qquad \four_1( \varphi_{\sigma_l})(v,w)= -\four_2( \varphi_{\sigma_l})(w,v).
\end{align} In particular, it follows that $\four_1(\varphi)(0)=\four_2(\varphi)(0)$ and $\four_1(\alpha)(0)=-\four_2(\alpha)(0)$. The vanishing of $\four_2( \varphi_\sigma)(0)$ is clear, and implies the vanishing of $\four_2(\varphi)(0)$. The vanishing of $\four_2( \varphi_{\sigma_l})(0)$ is also clear when $N>2$ since 
 \begin{align}
\overline{(vi+w)}_{\sigma_l}=(v_li+w_l)\prod_{\substack{m=1 \\ m \neq l}}^N (v_mi+w_m)^{\vert \sigma_l^{-1}(m) \vert}
\end{align} vanishes at $\vbf=0$, and is of degree $N-1 \geq 2$.  When $N=2$, then $\four_2( \varphi_{\sigma_l})(0)$ can be nonzero, since it can happen that $d=1$ in \eqref{fourier transform equality}. However, one can see that these contributions cancel out by writing out the form explicitly in this setting. See \cite[Proposition.~3.2]{rbrsln2}.

 Finally, when $N=1$ then $l=1$ and $\sigma_l$ is a function on an empty set. Hence, the form is
 \begin{align} \four_2( \varphi_{\sigma_l})(\vbf)& =-2i\pi (vi+w)\exp\left (-\pi  \lVert v\rVert^2-\pi \lVert w \rVert^2 \right ),
\end{align}
which also vanishes at $\vbf=0$.
\end{proof}


\subsection{Theta series} \label{sec:thetadiffforms} 

Using the Weil representation, we can write $\varphi_\sigma(\rho_g^{-1}\vbf)=\omega(g,1)\varphi_\sigma(\vbf)$ so that
    \begin{align}
    \varphi(z,\vbf)= 2^{-N}\pi^{-\frac{N}{2}}\sum_{\sigma}\omega(g_z,1)\varphi_\sigma(\vbf) \otimes \lambda(\sigma) \in \Omega^N(X) \otimes \Scal(V_\R)
\end{align}
and similarly
 \begin{align}
    \alpha(z,\vbf)=-2^{-N}\pi^{-\frac{N-1}{2}} \sum_{l=1}^N \sum_{\sigma_l}(-1)^{l-1} \omega(g_z,1)\varphi_{\sigma_l}(\vbf) \otimes \lambda(\sigma_l).
\end{align}
Using the fact that the partial Fourier transform intertwines the two representations, one can easily verify the following lemma. 
\begin{lem} \label{soinv}
The functions $\varphi_\sigma$ and $\varphi_{\sigma_l}$  satisfy
\begin{align}
\omega(1,k)\varphi_\sigma(\vbf)=j(k,i)^{-N}\varphi_\sigma(\vbf), \qquad \omega(1,k)\varphi_{\sigma_l}(\vbf)=j(k,i)^{-(N-2)}\varphi_{\sigma_l}(\vbf)
\end{align}
for $k \in \SO(2)$.
\end{lem}


For $\tau=x+iy \in \HH$ we define the theta series
\begin{align}
\Theta_{\varphi, \chi}(z,\tau) & \coloneqq j(h_\tau,i)^N  \sum_{\vbf \in V}  \chi(\vbf)\omega(1,h_\tau) \varphi(z,\vbf) \in \Omega^N(X)^\Gamma \otimes C^\8(\HH)
\end{align}
where $h_\tau \in \SL_2(\R)$ is any matrix sending $i$ to $\tau$. We can take
\begin{align}
h_\tau = \begin{pmatrix}
     \sqrt{y} & x\sqrt{y}^{-1} \\ 0 & \sqrt{y}^{-1}
 \end{pmatrix},
\end{align}
for which $j(h_\tau,i)=\sqrt{y}^{-1}$. Since we have multiplied by $j(h_\tau,i)^N$, Lemma \ref{soinv} implies that the function is independent of the choice of $h_\tau$. For $\gamma \in \Gamma'$ we have $h_{\gamma \tau}=\gamma h_\tau k$ for some $k \in \SO(2)$. It follows from the theta machinery that
\begin{align}
 \Theta_{\varphi, \chi}(z,\gamma\tau)=j(\gamma,\tau)^N \Theta_{\varphi, \chi}(z,\tau),
\end{align}
so the form transforms like a modular form of weight $N$ for $\Gamma' \subseteq \SL_2(\Z)$. Note that the form is not holomorphic but holomorphic in cohomology; see Theorem \ref{holo in coho}.

Similarly, we define the theta series
\begin{align}
\Theta_{\psi, \chi}(z,\tau) & \coloneqq j(h_\tau,i)^N  \sum_{\vbf \in V}  \chi(\vbf)\omega(1,h_\tau) \psi(z,\vbf) \in \Omega^N(X)^\Gamma \otimes C^\8(\HH)
\end{align}
which also transforms like a modular form of weight $N$ for $\Gamma'$. On the other hand, the theta series
\begin{align} \label{thetaalpha}
\Theta_{\alpha, \chi}(z,\tau) & \coloneqq j(h_\tau,i)^{-(N-2)} \sum_{\vbf \in V} \chi(\vbf)\omega(1,h_\tau)\alpha(z,\vbf)
\end{align} transforms like a modular form of weight $(N-2)$ for $\Gamma'$.
 \begin{rmk} In \cite{rbrsln2}, we use this form to construct a Borcherds lift that is adjoint to the derivative of $E_{\varphi,\chi}(z_1,\tau,s)$ defined below.
 \end{rmk}
 
\subsection{The Eisenstein class} \label{sec:eisensteinclass}

With respect to the splitting $X \simeq S \times \R_{>0}$ we can write
\begin{align}
    \Theta_{\varphi, \chi}(z,\tau) =\Theta_{\vptil, \chi}(z,\tau)+\Theta_{\psi, \chi}(z,\tau)  \frac{du}{u},
\end{align}
where we recall that $z=(z_1,u)$. We define the integral 
\begin{align} \label{integral definition}
E_{\varphi, \chi}(z_1,\tau,s) \coloneqq p_\ast \left (\Theta_{\varphi, \chi}(z,\tau) u^{-s}\right )= \int_{0}^\8 \Theta_{\psi, \chi}(z,\tau) u^{-s} \frac{du}{u}
\end{align}
for $s \in \C$.
\begin{prop} \cite[Proposition.6.14.1]{botu} \label{dpushforward} Let $d$ and $d_1$ be the exterior derivatives on $X$ and $S$ respectively. For a form $\eta \in \Omega_{\rd}^N(X)$ we have
\begin{align}
p_\ast( d \eta)=d_1 p_\ast (\eta) \in \Omega^{N-1}(S).
\end{align}
\end{prop}
%

\begin{prop} The integral \eqref{integral definition} converges uniformly on compact sets for any $s \in \C$. At $s=0$, the form $E_{\varphi, \chi}(z_1,\tau) \coloneqq E_{\varphi, \chi}(z_1,\tau,0)$ is closed.
\end{prop}

\begin{proof} We have
\begin{align}
\Theta_{\varphi, \chi}(z,\tau)-\four_2(\varphi)(0)\four_2(\chi)(0) & =O( e^{-C/u}) \qquad u \rightarrow 0, \\
\Theta_{\varphi, \chi}(g,h)-\four_1(\varphi)(0)\four_1(\chi)(0) & =O(e^{-Cu}) \qquad u \rightarrow \8.
\end{align}
By Proposition \ref{phisigma} we have $\four_1(\varphi)(0)=\four_2(\varphi)(0)=0$, so that $\Theta_{\varphi, \chi}(z,\tau)$ is rapidly decreasing at $0$ and $\8$ and the integral converges. Since $\Theta_{\varphi, \chi}(z,\tau)$ is closed, the pushforward $E_{\varphi, \chi}(z_1,\tau)$ is closed by Proposition \ref{dpushforward}.
%
%
\end{proof}

\begin{prop} \label{eisensteinseries}For $\re(s)>N$ we can write
\begin{align}
 E_{\varphi, \chi}(z_1,\tau,s) = \sum_{\sigma}  E_{\varphi, \chi}(g_{z_1},\tau,s)_\sigma \otimes \left ( \iota_{u \frac{\partial}{\partial u}} \lambda(\sigma) \right ),
\end{align}
where
\begin{align}
E_{\varphi, \chi}(g_{z_1},\tau,s)_\sigma \coloneqq \Lambda_N(s) y^{\frac{s}{2}} \sum_{\substack{\vbf \in V \\ \vbf \neq 0}} \four_2(\chi)(\vbf)  \frac{\left (g_{z_1}^{-1}(\overline{v\tau+w}) \right )_\sigma}{\lVert g_{z_1}^{-1}(v\tau+w) \rVert^{2N+s}}
\end{align}
with $g_{z_1} \in \SL_N(\R)$ such that $z_1=g_{z_1}g_{z_1}^t$, and 
\begin{align}
\Lambda_N(s) \coloneqq \Gamma \left ( N+\frac{s}{2}\right ) \frac{(-1)^{N-1}i^{N}}{2\pi^{N+\frac{s}{2}}}.
\end{align}
\end{prop}
\begin{proof}
Using Lemma \ref{lem intertwiner} and Poisson summation, we write
\begin{align}
E_{\varphi, \chi}(z_1,\tau,s)& = j(h_\tau,i)^{N} \int_{0}^\8 \sum_{\vbf \in V} \chi(\vbf)  \omega(1,h_\tau)\psi(z,\vbf)u^{-s} \frac{du}{u}\nonumber \\
& = j(h_\tau,i)^{N} \sum_{\vbf \in V} \four_2(\chi)(\vbf) \int_{0}^\8 \omega'(1,h_\tau)\four_2(\psi)(z,\vbf)u^{-s} \frac{du}{u},
\end{align}
where the change of sum and integral is valid for $\re(s)>N$. Using the explicit formula for $\psi$ in Proposition \ref{explicitpsi} and for $\four_2(\varphi_\sigma)$ in Proposition \ref{phisigma}, we find that
\begin{align*}
   & j(h_\tau,i)^{N}\omega'(1,h_\tau)\four_2(\psi)(z,\vbf) \nonumber \\
    & \hspace{2cm} = \frac{(-1)^{N-1}i^N}{y^N} \sum_{\sigma}  \frac{\left (g_{z_1}^{-1}(\overline{v\tau+w}) \right )_\sigma}{u^{2N}} e^{-\frac{\pi}{yu^{2}}\lVert g_{z_1}^{-1}(v\tau+w) \rVert^2}\otimes \left (  \iota_{u \frac{\partial}{\partial u}} \lambda(\sigma)  \right )
\end{align*}
so that for $\vbf \neq 0$ we have
\begin{align} \label{sigmhatform}
    & j(h_\tau,i)^{N}\int_{0}^\8 \omega'(1,h_\tau)\four_2(\psi)(z,\vbf) u^{-s} \frac{du}{u}  \nonumber \\
    & \hspace{2cm} \qquad \qquad=\Lambda_N(s)  \sum_{\sigma}  \frac{y^{\frac{s}{2}}  \left(g_{z_1}^{-1}(\overline{v\tau+w}) \right )_\sigma}{\lVert g_{z_1}^{-1}(v\tau+w) \rVert^{2N+s}} \otimes \left ( \iota_{u \frac{\partial}{\partial u}} \lambda(\sigma)  \right ).
\end{align}
\end{proof}


\subsection{The Fourier coefficients of the Eisenstein class} \label{sec:fourierexpansion} The Weil representation acts by
\begin{align}   \omega(1,h_\tau)\varphi(z,\vbf)=\sqrt{y}^{-N}\varphi^0(z,\sqrt{y}\vbf)q^{Q(\vbf)},
\end{align}
so that
\begin{align}
\Theta_{\psi, \chi}(z,\tau) & = \sum_{\vbf \in V}  \chi(\vbf) \psi^0(z,\sqrt{y}\vbf)q^{Q(\vbf)} \in \Omega^N(X)^\Gamma \otimes C^\8(\HH)^{\Gamma'}.
\end{align}
For a class $[\vbf] \in \Gamma \backslash V$ we set
\begin{align}
E_{[\vbf]}(z_1,y)\coloneqq \int_0^\8 \sum_{\vbf' \in \Gamma \vbf}  \psi^0(z,\sqrt{y}\vbf')\frac{du}{u} \in \Omega^{N-1}(S_\Gamma) \otimes C^\8(\R_{>0}).
\end{align}
By the estimate given in Lemma \ref{lemmapolynomial}, the form $\int_0^\8\psi^0(z,\vbf)\frac{du}{u}$ is rapidly decreasing as $\vbf$ is summed over regular vectors of fixed length $Q(\vbf)$. Thus, if $\vbf$ is regular we can change the order of integration and summation to write
\begin{align}
E_{[\vbf]}(z_1,y)=  \sum_{\vbf' \in \Gamma \vbf}  \int_0^\8 \psi^0(z,\sqrt{y}\vbf')\frac{du}{u},
\end{align}
whereas for the singular vectors one needs to add $u^{-s}$ to change the order of integration and summation. We can rewrite the sum
\begin{align}
\int_0^\8 \Theta_{\psi, \chi}(z,\tau)\frac{du}{u} & = \int_0^\8 \sum_{n \in \Z} \sum_{\substack{\vbf \in \Gamma \backslash V \\ Q(\vbf)=n}} \sum_{\vbf' \in \Gamma \vbf}  \chi(\vbf) \psi^0(z,\sqrt{y}\vbf)q^{Q(\vbf)} \frac{du}{u} \\
&=  \sum_{n \in \Z} \left ( \sum_{\substack{\vbf \in \Gamma \backslash V \\ Q(\vbf)=n}} \chi(\vbf) E_{[\vbf]}(z_1,y) \right )q^n.
\end{align}

Thus, the $n$-th Fourier coefficient of $E_{\varphi, \chi}(z_1,\tau)$ is given by
\begin{align}
a_n(E_{\varphi, \chi}(z_1,\tau))=\sum_{\substack{\vbf \in \Gamma \backslash V \\ Q(\vbf)=n}}\chi(\vbf)  E_{[\vbf]}(z_1,y).
\end{align}
For the constant term, we will consider the regular and singular vectors separately. Since $\psi(z,0)=0$, we can remove $0$ from the summation and sum over $\Lcal^\vee \smallsetminus \{0\}$. We write
\begin{align*}
\Lcal^\vee \smallsetminus \{0\}=\Lcal^\vee_{\sing,1} \sqcup \Lcal^\vee_{\sing,2} \sqcup \Lcal^\vee_{\reg}
\end{align*}
where
\begin{align*}
\Lcal^\vee_{\sing,2} \coloneqq  \left \{ \vbf \in \Lcal^\vee  \ \vert \ v = 0, \ w \neq 0\right \}, & \quad
\Lcal^\vee_{\sing,1} \coloneqq \left \{ \vbf \in \Lcal^\vee  \ \vert \ v\neq0, \ w = 0\right \}, \\
\Lcal^\vee_\reg \coloneqq &\left \{ \vbf \in \Lcal^\vee  \ \vert \ v\neq0, \ w \neq 0\right \}.
\end{align*}
The constant term splits as
\begin{align} \label{eq constant term}
a_0(E_{\varphi, \chi}(z_1,\tau)) & =\sum_{\vbf \in \Gamma \backslash \Lcal^\vee_{\sing,1} } \chi(\vbf) E_{[\vbf]}(z_1,y) + \sum_{\vbf \in \Gamma \backslash \Lcal^\vee_{\sing,2}} \chi(\vbf) E_{[\vbf]}(z_1,y) \nonumber \\
& + \sum_{\substack{\vbf \in \Gamma \backslash \Lcal^\vee_{\reg} \\ Q(\vbf)=0}} \chi(\vbf) E_{[\vbf]}(z_1,y).
\end{align}

In the next section, we will show that the  regular orbits $E_{[\vbf]}(z_1,y)$ with $Q(\vbf) \leq 0$ are trivial in cohomology.

\subsubsection{Pairings}We have a pairing
\begin{align}
  H_c^{{(N^2-N)/2}}(S_\Gamma;\C) \times H^{N-1}(S_\Gamma;\C)  \longrightarrow \C,
\end{align}
defined by
\begin{align}
 ( \omega, \omega' )=\int_{S_\Gamma} \omega \wedge \omega'.
\end{align}
On the other hand, by Poincaré duality we have $H_c^{{(N^2-N)/2}}(S_\Gamma;\C) \simeq H_{N-1}(S_\Gamma;\C)$. Hence, the pairing is
\begin{align}
H_{N-1}(S_\Gamma;\C) \times H^{N-1}(S_\Gamma;\C) \longrightarrow \C,
\end{align}
where
\begin{align}
 ( Z, \omega')=(-1)^{\frac{N(N-1)^2}{2}}\int_Z \omega'.
\end{align}
The Poincaré dual  $\omega_Z$ in $H_c^{{(N^2-N)/2}}(S_\Gamma;\C)$ of a class $Z$ in $H_{N-1}(S_\Gamma;\C)$ is such that
\begin{align}
 ( Z, \omega')=(-1)^{\frac{N(N-1)^2}{2}} \int_Z \omega'=\int_{S_\Gamma} \omega_Z \wedge \omega'=[ \omega_Z, \omega']
\end{align}
for any $\omega' \in H^{N-1}(S_\Gamma;\C)$. Similarly, by Poincaré duality we have $H^{N-1}(S_\Gamma;\C) \simeq H^{\BM}_{(N^2-N)/2}(S_\Gamma;\C)$, where the Borel-Moore homology is the homology of locally finite chains. Thus, the pairing becomes
\begin{align}
H_c^{{(N^2-N)/2}}(S_\Gamma;\C) \times H^{\BM}_{(N^2-N)/2}(S_\Gamma;\C) \longrightarrow \C,
\end{align}
where
\begin{align}
 ( \omega,Z')=\int_{Z'} \omega.
\end{align}
Finally, combining the two gives the intersection pairing in homology
\begin{align}
H_{N-1}(S_\Gamma;\C) \times H^{\BM}_{(N^2-N)/2}(S_\Gamma;\C) \longrightarrow \C,
\end{align}
where
\begin{align}
 (Z, Z') &=(-1)^{\frac{N(N-1)^2}{2}}\int_Z \omega_{Z'} =\int_{Z'} \omega_Z=\int_{S_\Gamma} \omega_Z \wedge \omega_{Z'}=(\omega_Z,\omega_{Z'}).
\end{align}
Note that we have $(Z, Z')=(-1)^{\frac{N(N-1)^2}{2}}( Z', Z)$.

\subsection{The nonpositive regular orbits} \label{subsec:negative fourier}
\begin{prop}\cite[Proposition.~6.15]{botu} \label{projection}
    For $\eta \in \Omega^k_{\rd}(X)$ and $\omega \in \Omega_c^{\dim(S)-k+1}(S)$ we have the projection formula
    \begin{align}
      \int_S \omega \wedge \left   ( p_\ast \eta \right )= \int_X  ( p^\ast \omega) \wedge \eta.  
    \end{align}
\end{prop}

Combining with Proposition \ref{thomformprop2}, we deduce the following.
\begin{prop} \label{thomformprop3} Let $\vbf$ be a regular vector. 
\begin{enumerate} \item  If $Q(\vbf)>0$, then the form $p_\ast(\varphi^0(z,\vbf))$ is a Poincaré dual to $\Gamma_\vbf \backslash S_\vbf \subset \Gamma_\vbf \backslash S$. Thus, if $\omega \in \Omega_c^{(N^2-N)/2}(\Gamma_\vbf \backslash S)$ is a compactly supported form, then
\begin{align} 
\int_{\Gamma_\vbf \backslash S} \omega \wedge p_\ast(\varphi^0(z,\vbf))=\int_{\Gamma_\vbf \backslash S_\vbf} \omega.
\end{align}
\item  If $Q(\vbf) \leq 0$, then the form $p_\ast \left ( \varphi^0(z,\vbf) \right )$ is exact.
\end{enumerate}
\end{prop}
\begin{proof}
The first follows from (1) of Proposition \ref{thomformprop2} and the projection formula in Proposition \ref{projection}, since
\begin{align} 
\int_{\Gamma_\vbf \backslash S} \omega \wedge p_\ast(\varphi^0(z,\vbf)) & = \int_{\Gamma_\vbf \backslash X} p^\ast (\omega) \wedge \varphi^0(z,\vbf) \\
&=\int_{\Gamma_\vbf \backslash X_\vbf} p^\ast (\omega) \\
&=\int_{\Gamma_\vbf \backslash S_\vbf} \omega.
\end{align}
In the last step, we used that the projection restricts to a diffeomorphism $X_\vbf \longrightarrow S_\vbf$. Note that $p^\ast (\omega)$ is not compactly supported in $\Gamma_\vbf \backslash X$, as in Proposition \ref{thomformprop2}, since it constant in $u$. However, since $\vbf$ is regular, the form $\varphi^0(z,\vbf)$ is rapidly decreasing as $u \rightarrow 0$ and $\8$, and the integral against $p^\ast (\omega)$ is convergent.

By Proposition \ref{thomformprop2} we have \begin{align}
\varphi^0(z,\vbf)=d\int_1^\8 \alpha^{0}(z,\sqrt{t}\vbf) \frac{dt}{t}
\end{align}
for negative regular vector. Thus, the second part of the proposition follows from Proposition \ref{dpushforward}. Note that $\int_1^\8 \alpha^{0}(z,\sqrt{t}\vbf) \frac{dt}{t}$ is also rapidly decreasing as $u$ goes to $0$ and $\8$.
\end{proof}

Since $p_\ast\left ( \varphi^0(z,\vbf) \right )$ is exact if $\vbf$ is regular and negative, we deduce the following.

\begin{prop}  \label{exactform} If $[\vbf]$ is a regular orbit with $Q(\vbf) \leq 0$, then
\begin{align}
    \int_{S_\Gamma} \omega \wedge E_{[\vbf]}(z_1,y)=0.
\end{align}
\end{prop}
%

\subsection{The positive regular orbits} \label{subsec:positive fourier} 
%
%
%
%
%

The following proposition implies that
\begin{align}
a_n(E_{\varphi, \chi}(z_1,\tau))=\sum_{\substack{\vbf \in \Gamma \backslash V \\ Q(\vbf)=n}}\chi(\vbf)  E_{[\vbf]}(z_1,y).
\end{align}
represents the Poincaré dual to $Z_n(\chi)$.

\begin{prop} \label{fouriercoeffn} If $[\vbf]$ is a regular orbit in $V$ with $Q(\vbf)>0$, then the form $E_{[\vbf]}(z_1,y)$ represents the Poincaré dual to $Z_{[\vbf]}$ in $H^{N-1}(S_\Gamma)$. In other words, for every compactly supported form $\omega \in \Omega_c^{(N^2-N)/2}(S_\Gamma)$ we have
\begin{align}
   \int_{S_\Gamma} \omega \wedge E_{[\vbf]}(z_1,y)=\int_{Z_{[\vbf]}} \omega.
\end{align}
\end{prop}

\begin{proof} It follows from a standard unfolding argument as in \cite{kmcjm}. See also \cite[Lemma.~2.1]{KM82}. Let $F \subset S$ be a fundamental domain for $\Gamma$ in $S$, and $F_\vbf=\cup_{\gamma \in \Gamma_\vbf \backslash \Gamma} F$ a fundamental domain for $\Gamma_\vbf$. We view $\omega$ as $\Gamma$-invariant form with support contained in $F$. 
 and unfolding the integral gives
\begin{align}
\int_{F} \omega \wedge E_{[\vbf]}(z_1,y) & = \int_{F} \omega \wedge \left ( \sum_{\gamma \in \Gamma_\vbf \backslash \Gamma} p_\ast \varphi^0(z,\rho_{\gamma}^{-1}\vbf) \right ) \\
&= \sum_{\gamma \in \Gamma_\vbf \backslash \Gamma}  \int_{F} \omega \wedge \left ( p_\ast \varphi^0(z,\rho_{\gamma}^{-1}\vbf) \right ) \\
& = \sum_{\gamma \in \Gamma_\vbf \backslash \Gamma}  \int_{F}  \gamma^\ast \left ( \omega \wedge p_\ast \varphi^0(z,\vbf) \right ) \\
 & = \sum_{\gamma \in \Gamma_\vbf \backslash \Gamma} \int_{\gamma F} \omega \wedge p_\ast \varphi^0(z,\vbf) \\
& = \int_{F_\vbf} \omega \wedge p_\ast \varphi^0(z,\vbf) \\
\end{align}
where in the second step we can bring the integral inside the sum, since the sum of $\varphi^0(z,\vbf)$ over regular orbit is rapidly decreasing by Lemma \ref{lemmapolynomial}. By Proposition \ref{thomformprop3},  we then deduce
\begin{align}
 \int_{F_\vbf} \omega \wedge p_\ast \varphi^0(z,\vbf) & = \int_{F_\vbf \cap S_\vbf} \omega = \int_{Z_{[\vbf]}} \omega.
\end{align}

\end{proof}

\subsection{The constant term}
\label{eulerclass}
In this section, we show that the constant term is given by the Eisenstein classes constructed by Bergeron-Charollois-Garcia. Let $L$ be a lattice in $\Z^N$ as previously. The quotient 
\begin{align}
 \Tcal \coloneqq \Gamma \backslash (S \times T) & \longrightarrow \Gamma \backslash S
\end{align}
is a torus bundle with fibers $T\coloneqq \R^N/L$, where $\Gamma$ acts by $\gamma(z,v+L)=(\gamma z,\gamma v+L)$.

\subsubsection{Topological Eisenstein classes} We recall the construction of the canonical Eisenstein classes following Bergeron-Charollois-Garcia \cite{bcg}. Similar classes have been previously considered by Bismut-Cheeger \cite{bc}, and Faltings \cite{faltings} for Siegel spaces. Note that the space $X$ (resp. $S^+$) in {\em loc. cit.} is the space $S$ (resp. $X$) in this paper. Let $m$ be an integer coprime to $\vert D_L\vert$ and let $\Tcal[m]=\Gamma \backslash (S \times T[m])$ be the $m$-torsion points on $\Tcal$, and $\Tcal_0=\Gamma \backslash (S \times \{0\})$ the image of the zero section. The Thom isomorphism induces an isomorphism
\begin{align}
H^0(\Tcal[m]) \longrightarrow H^N(\Tcal,\Tcal \smallsetminus \Tcal[m]).
\end{align}
Sullivan \cite{S75} shows that the image of the class
\begin{align}
[\Tcal[m]-m^N\Tcal_0]=0 \in H^N(\Tcal,\Tcal \smallsetminus \Tcal[m])
\end{align}
is trivial in cohomology. The long exact sequence  
\begin{equation}
\hspace*{-2.5cm}  \begin{tikzcd}
H^{N-2}(\Tcal \smallsetminus \Tcal[m]) \rar[] & H^{N-1}(\Tcal,\Tcal \smallsetminus \Tcal[m]) \rar[] & H^{N-1}(\Tcal)\ar[ out=-10, in=170]{dll} \\
H^{N-1}(\Tcal \smallsetminus\Tcal[m]) \rar[] & H^N(\Tcal,\Tcal \smallsetminus \Tcal[m]) \rar[] & H^N(\Tcal)
\end{tikzcd}    
\end{equation}
shows that it has a preimage in $H^{N-1}(\Tcal \smallsetminus\Tcal[m])$, only determined up to  an element in $H^{N-1}(\Tcal)$.

To find a canonical representatives, let $a$ be an integer coprime to $m$. It induces a finite covering map $\Tcal \longrightarrow \Tcal$ and a pullback 
\begin{align}
a^\ast \colon H^N(\Tcal,\Tcal\smallsetminus\Tcal[m]) \longrightarrow H^N (\Tcal,\Tcal \smallsetminus \Tcal[am]).
\end{align} It also defines a pushforward
\begin{align}
a_\ast \colon \Omega^N(S \times T) \longrightarrow \Omega^N(S \times T), \quad a_\ast(\omega)(z_1,v) =\sum_{\substack{w \in T \\ aw=v}} \omega(z,w)
\end{align}
by summing over the fiber. The map is $\Gamma$-equivariant, and
\begin{align}
a_\ast \left ( \Omega^{N-1}(S \times (T\smallsetminus T[am])\right ) \subset \Omega^{N-1}(S \times (T\smallsetminus T[m]).
\end{align}  Thus, it induces a pushforward map
\begin{align}
a_\ast \colon H^N(\Tcal,\Tcal \smallsetminus \Tcal[am]) \longrightarrow H^N(\Tcal,\Tcal \smallsetminus \Tcal[m]).
\end{align}
Finally, the pullback by the inclusion $T[m]\subset T[am]$ is
\begin{align}
i^\ast \colon H^N(\Tcal \smallsetminus \Tcal[m]) \longrightarrow H^N(\Tcal \smallsetminus \Tcal[am]).
\end{align} 
\begin{thm}\cite[Section.~3]{bcg}There is a unique class 
\begin{align}
z_m \in H^{N-1}(\Tcal\smallsetminus \Tcal[m];\Q)
\end{align} such that $a_\ast \iota^\ast z_m=z_m$ for all integers $a$ coprime to $m$.
\end{thm}
Let $v_0 \in D_L$ be nonzero {\em i.e.} represented by a vector $v_0 \in L^\ast$ not in $L$. Since the order is coprime to $m$, it defines a section  
\begin{align}
v_0 \colon S_\Gamma \longrightarrow \Tcal \smallsetminus \Tcal[m].
\end{align}
The canonical Eisenstein class is the class
\begin{align}
z(v_0) \coloneqq \frac{1}{m^N-1}v_0^\ast z_m \in H^{N-1}(S_\Gamma;\Q),
\end{align}
where the normalization by $m^N-1$ is chosen so that the class does not depend on $m$.

\subsubsection{De Rham representative} We have an isomorphism
\begin{align} \label{iso bundle}
 X \times \R^N \longrightarrow E, \qquad (z,v)  \longrightarrow [g,g^{-1}v]
\end{align}
where $z=gg^t$. Let $\Phi \in \Omega^N(X \times \R^N)$ be the pullback of the Thom form $U$ by this map. Let us denote by $t_{v}$ the translation $t_{v} \colon X \times V \longrightarrow V$ defined by $t_{v}(z,v')=(z,v+v')$.  The form
\begin{align}
\Theta_\Phi(z,v') \coloneqq \sum_{v \in L} t_{v}^\ast \Phi(z,v') \in \Omega^N(X \times T)^{\Gamma}
\end{align}
is closed and $\Gamma$-invariant. By \cite[Proposition.~16]{bcg},  it is a Thom form of the torus bundle. We write
\begin{align}
\Phi(z,v')=\widetilde{\Phi}(z,v')+\Psi(z,v')\frac{du}{u}
\end{align}
where $\Phi \in \Omega^{N}(X \times T)$ and $\Psi \in \Omega^{N-1}(X \times T)$ are of degree $0$ along $\R_{>0}$. Similarly, we define
\begin{align}
\Theta_\Psi(z,v') \coloneqq \sum_{v \in L} t_{v}^\ast \Psi(z,v') \in \Omega^{N-1}(X \times T)^{\Gamma}.
\end{align}
Since the exponential factor of $\Psi$ is of the form $\exp(-u^{-1}\lVert v'+v\rVert^2)$, the form $\Theta_\Psi(z,v')$ is rapidly decreasing as $u \rightarrow 0$ if $v'$ is not in $L$. Hence, the integral
\begin{align}
E_{\Psi}(z_1,v',s) \coloneqq \int_0^\8 \Theta_\Psi(z,v') u^{-s}\frac{du}{u} \in \Omega^{N-1}(\Tcal-\Tcal_0)
\end{align}
converges for $\re(s)>0$ and away from the zero section. It has a meromorphic continuation, with at most a simple pole at $s=1$. This is the form defined by Bergeron-Charollois-Garcia in  \cite[section.~8.5]{bcg}. 

Let $v_0$ be a vector in $L^\ast$ and suppose that $\Gamma \subset \SL_N(\Z)$ preserves the class of $v_0 \in D_L$. Let us also denote by $v_0$ the corresponding section 
\begin{align}
S_\Gamma \longrightarrow \Tcal \smallsetminus \Tcal_0, \quad z \longmapsto (z,v_0).
\end{align}
The pullback defines a closed form
\begin{align}
E_{\Psi,v_0}(z_1,s)\coloneqq v_0^\ast E_\Psi(z_1,v',s) \in \Omega^{N-1}(S_\Gamma).
\end{align}

\begin{thm}  \cite[Theorem.~23]{bcg} The class $-z(v_0)$ in $H^{N-1}(S_\Gamma,\Q)$ is represented by the form $E_{\Psi,v_0}$.
\end{thm}
For $v,w \in \Q^N$ let us define the test functions $\chi_1(v) \coloneqq \chi(v,0)$ and $\chi_2(w) \coloneqq \chi(0,w)$. Let
\begin{align}
z(\chi_i)\coloneqq \sum_{v_0 \in D_L}\chi_i(v_0)z(v_0).
\end{align}
Let $\widehat{\chi}_2$ be the finite Fourier transform of $\chi_2$ and set
\begin{align}
z(\chi)\coloneqq -z(\chi_1)-z(\widehat{\chi}_2).
\end{align}

\begin{prop} \label{prop eulerclass} We have
\begin{align}
a_0(E_{\varphi, \chi}(z_1,\tau))=z(\chi) \ \textrm{in} \ H^{N-1}(S_\Gamma,\Q).
\end{align}
\end{prop}

\begin{proof} We have seen in \eqref{eq constant term} that the constant term is
\begin{align} 
a_0(E_{\varphi, \chi}(z_1,\tau)) & =\sum_{\vbf \in \Gamma \backslash \Lcal^\vee_{\sing,1} } \chi(\vbf) E_{[\vbf]}(z_1,y) + \sum_{\vbf \in \Gamma \backslash \Lcal^\vee_{\sing,2}} \chi(\vbf) E_{[\vbf]}(z_1,y) \nonumber \\
& + \sum_{\substack{\vbf \in \Gamma \backslash \Lcal^\vee_{\reg} \\ Q(\vbf)=0}} \chi(\vbf) E_{[\vbf]}(z_1,y),
\end{align}
where, by Proposition \ref{exactform}, the regular term is exact.

Note that the following diagram commutes
\begin{equation}
\begin{tikzcd}[column sep=large,row sep=large]
X\times \R^N \arrow[r] & G \times_K \R^N  \\
X \arrow[u, "v"] \arrow[ru, swap, "s_{(v,0)}"] &
\end{tikzcd}
\end{equation}
where the top arrow is the map \eqref{iso bundle} and $s_{\vbf}(z)=\left [g_z, (g_z^{-1}v-g_z^tw)/\sqrt{2} \right ]$ is the section used to construct $\varphi^0(z,\vbf)=s_\vbf^\ast U$. Hence, we have $v^\ast \Phi(z,v')=\varphi^0(z,(v,0))$. It follows that 
\begin{align}
v_0^\ast \Theta_\Phi(z,v') & = \sum_{v \in L} v_0^\ast t_{v}^\ast \Phi(z,v')  \nonumber \\ 
& = \sum_{v \in L} (v_0+v)^\ast \Phi(z,v') \nonumber \\
& = \sum_{v \in v_0+L} \varphi^0(z,(v,0)) \nonumber \\
& = \sum_{v \in \Gamma \backslash (v_0+L)} \sum_{v' \in \Gamma v} \varphi^0(z,(v',0)).
\end{align}
Hence, taking $\frac{du}{u}$ components gives
\begin{align}
v_0^\ast \Theta_\Psi(z,v') & = \sum_{v \in \Gamma \backslash (v_0+L)} \sum_{v' \in \Gamma v} \psi^0(z,(v',0))
\end{align}
and the integral over $\R_{>0}$ is
\begin{align}
E_{\Psi,v_0}(z_1,s)& = v_0^\ast \int_0^\8 \Theta_\Psi(z,v') u^{-s}\frac{du}{u} \nonumber \\
&= \int_0^\8 \sum_{v \in \Gamma \backslash (v_0+L)} \sum_{v' \in \Gamma v} \psi^0(z,(v',0)) u^{-s} \frac{du}{u} \nonumber \\
& = \sum_{v \in \Gamma \backslash (v_0+L)} E_{[(v,0)]}(z_1,1,s).
\end{align}
Finally, note that the function $E_{[(v,0)]}(z_1,y)$ satisfies the homogeneity property 
\begin{align}
E_{[(v,0)]}(z_1,y,s)=y^sE_{[(v,0)]}(z_1,1,s),
\end{align}which can be seen from the change of variable $t=uy$, and shows that at $s=0$
\begin{align}
\sum_{[v] \in \Gamma \backslash (v_0+L) } E_{[(v,0)]}(z_1,1)=E_{\Psi,v_0}(z_1).
\end{align}
It follows that
\begin{align}
\sum_{\vbf \in \Gamma \backslash \Lcal^\vee_{\sing,1} } \chi(\vbf) E_{[\vbf]}(z_1,y) =\sum_{v_0 \in D_L} \chi_1(v_0) E_{\Psi,v_0}
\end{align}
represents $-z(\chi_1)$, and the second singular term follows from Poisson summation.
\end{proof}

\subsection{Holomorphic theta lift} \label{sec:thetalift}
 The form $E_{\varphi, \chi}(z_1,\tau)$ transforms like a modular of weight $N$. To obtain a lift to {\em holomorphic} modular forms, we have to show that the forms vanishes (in cohomology) under the operator $\frac{\partial}{\partial \overline{\tau}}=\frac{1}{2} \left (\frac{\partial}{\partial x}+i\frac{\partial}{\partial y} \right )$.

Recall that we defined in \eqref{thetaalpha} the theta series
\begin{align}
\Theta_{\alpha, \chi}(z,\tau) = y\sum_{\vbf \in V} \chi(\vbf)\alpha^0(z,\sqrt{y}\vbf)q^{Q(\vbf)}.
\end{align}

Since $\Theta_{\alpha, \chi}$ is rapidly decreasing at $0$ and $\8$, the integral
 \begin{align}
 E_{\alpha, \chi}(z_1,\tau) \coloneqq p_\ast \left ( \Theta_{\alpha, \chi}(z,\tau) \right )\in \Omega^{N-2}(S_\Gamma) \otimes C^\8(\HH)
 \end{align}
along the fiber of $X \longrightarrow S$ converges.

\begin{thm} \label{holo in coho}
The form $ E_{\alpha, \chi}(z_1,\tau)$ transforms like a modular form of weight $N-2$ and 
\begin{align}
2iy^2\frac{\partial}{\partial \overline{\tau}}E_{\varphi, \chi}(z,\tau)=dE_{\alpha, \chi}(z_1,\tau).
\end{align} In particular, the lift $E_{\varphi, \chi}(C,\tau)$ is holomorphic.
\end{thm}

\begin{proof} The form $\alpha^0$ satisfies
\begin{align}
d\alpha^0(z,\sqrt{y}\vbf)=y\frac{\partial}{\partial y}\varphi^0(z,\sqrt{y}\vbf)=2iy \frac{\partial}{\partial \bar{\tau}}\varphi^0(z,\sqrt{y}\vbf). 
\end{align}
Hence, we have
\begin{align}
d\Theta_{\alpha, \chi}(z,\tau)
& = y\sum_{\vbf \in \Lcal^\vee} \chi(\vbf)d\alpha^0(z,\sqrt{y}\vbf)q^{Q(\vbf)} \nonumber \\
& = 2iy^2 \frac{\partial}{\partial \bar{\tau}} \Theta_{\varphi, \chi}(z,\tau).
\end{align}
Since $\Theta_{\alpha, \chi}(z,\tau)$ is rapidly decreasing as $u$ goes to $0$ and $\8$, we deduce from Proposition \ref{dpushforward} that
\begin{align}
2iy\frac{\partial}{\partial \overline{\tau}} E_{\varphi, \chi}(z_1,\tau)=p_\ast \left (d \Theta_{\alpha, \chi}(z,\tau)\right )=d_1 p_\ast \left (\Theta_{\alpha, \chi}(z,\tau)\right ).
\end{align}
\end{proof}

For a compactly supported form $\omega \in \Omega_c^{(N^2-N)/2}(S_\Gamma)$, we define
\begin{align} \label{pairing theta lift}
E_{\varphi, \chi}(\omega,\tau) \coloneqq  \left ( \omega, E_{\varphi, \chi}(z_1,\tau) \right ) = \int_{S_\Gamma} \omega \wedge E_{\varphi, \chi}(z_1,\tau) .
\end{align}
By Poincaré duality, we have $H_c^{(N^2-N)/2}(S_\Gamma;\C) \simeq H_{N-1}(S_\Gamma;\C)$. If $\omega=\omega_C$ is a form in $ \Omega_c^{(N^2-N)/2}(S_\Gamma;\C)$ representing the Poincaré dual of a cycle $C$, then we write
\begin{align}
E_{\varphi, \chi}(C,\tau) \coloneqq  (-1)^{\frac{N(N-1)^2}{2}} E_{\varphi, \chi}(\omega_C,\tau)= \int_{C} E_{\varphi, \chi}(z_1,\tau).
\end{align}
Let $\Mcal_N(\Gamma')$ be the space of holomorphic modular forms of weight $N$ and level $\Gamma' \subset \SL_2(\Z)$. We combine sections  \ref{subsec:negative fourier} and  \ref{subsec:positive fourier} to deduce the following theorem.
\begin{thm} The form $E_{\varphi, \chi}(z_1,\tau)$ defines a lift
\begin{align}
E_{\varphi, \chi} \colon H_{N-1}(S_\Gamma;\Z) \longrightarrow \Mcal_N(\Gamma'), \qquad
Z \longmapsto E_{\varphi, \chi}(Z),
\end{align}
where for a cycle $Z \in Z_{N-1}(S_\Gamma, \Z)$ we have 
\begin{align}
E_{\varphi, \chi}(Z,\tau)=\int_Z z(\chi)+ (-1)^{\frac{N(N-1)^2}{2}}\sum_{n=1}^\8\left ( Z,Z_n(\chi) \right ) q^n. 
\end{align}
\end{thm}
 

\subsection{Span of special cycles} We also get a dual lift
\begin{align}
 \Lambda_{\varphi, \chi} \colon \Scal_N(\Gamma') \longrightarrow   H^{N-1}(S_\Gamma;\C), \qquad
 f \longmapsto \Lambda_{\varphi, \chi}(f)
\end{align}
where for a weight $N$ cusp form $f \in \Scal_N(\Gamma')$ we define
\begin{align}
\Lambda_{\varphi, \chi}(z_1,f) & \coloneqq  \langle E_{\varphi, \chi}(z_1,\tau),f\rangle_{\pet} \nonumber \\
& = \int_{\Gamma'\backslash \HH} E_{\varphi, \chi}(z_1,\tau) \overline{f(\tau)}y^{N-2} dxdy, \qquad z_1\in S_\Gamma.
\end{align}
Since the integrals are absolutely convergent, it follows from Fubini's theorem that
\begin{align}
\left ( \omega, \Lambda_{\varphi, \chi}(f) \right )= \langle E_{\varphi, \chi}(\omega),f\rangle_{\pet}
\end{align}

Recall that we have a pairing
\begin{align}
(-,- ) \colon H_{N-1}(S_\Gamma;\C) \times H^{\BM}_{(N^2-N)/2}(S_\Gamma;\C) \longrightarrow \C.
\end{align}
For a subspace $W \subset H^{\BM}_{(N^2-N)/2}(S_\Gamma;\C)$ we denote its orthogonal complement by
\begin{align}
W^\perp \coloneqq \left \{ Z \in H_{N-1}(S_\Gamma;\C)  \ \vert  \ ( Z,Z')=0 \ \textrm{for all} \ Z' \in W \right \}.
\end{align}
The following theorem is analogous to  \cite[Theorem.~4.5]{kmcjm}. Note that here we extend the coefficients to $\C$ to get a lift
\begin{align}
E_{\varphi, \chi} \colon H_{N-1}(S_\Gamma;\C) \longrightarrow \Mcal_N(\Gamma').
\end{align}
\begin{thm} \label{theoremC416} We have
\begin{align}
\ker(E_{\varphi, \chi}) \simeq \Span \{Z_n(\chi) \ \vert \ n \in \NN \}^\perp \subset H_{N-1}(S_\Gamma;\C).
\end{align}
Moreover,  under Poincaré duality $H^{N-1}(S_\Gamma;\C) \simeq H^{\BM}_{(N^2-N)/2}(S_\Gamma;\C)$ we also have
\begin{align}
\image(\Lambda_{\varphi, \chi}) \subset \Span \{Z_n(\chi) \ \vert \ n \in \NN \}.
\end{align}
\end{thm}

\begin{proof} The first follows from the definition of the theta lift.  The lift $E_{\varphi, \chi}(Z)$ vanishes if and only if all the Fourier coefficients $[ Z,Z_n(\chi) ]$ vanishes, {\em i.e.} when 
\begin{align}
Z \in \{Z_n(\chi) \ \vert \ n \in \NN \}^\perp.
\end{align}

On the other hand, if $\omega$ is an element of $\ker(E_{\varphi, \chi})$, then for $\eta=\Lambda_{\varphi, \chi}(f)$ in the image of $\Lambda_{\varphi, \chi}$ we have
\begin{align}
( \omega, \eta )= ( \omega, \Lambda_{\varphi, \chi}(f) ) = \langle E_{\varphi, \chi}(\omega), f \rangle_{\pet}=0.
\end{align}
This shows that $\ker(E_{\varphi, \chi}) \subset \image(\Lambda_{\varphi, \chi})^\perp$, or equivalently $\image(\Lambda_{\varphi, \chi})\subset \ker(E_{\varphi, \chi})^\perp $ 
\end{proof}

\begin{rmk}
Equivalently, we have $\ker(\Lambda_{\varphi, \chi})=\image(E_{\varphi, \chi}) \cap \Scal_N(\Gamma').$
\end{rmk}

\subsection{Periods over tori} \label{sec:periodtori} Let $T(\Q) \subset \GL_N(\Q)$ be a torus, and $T^{(1)}(\Q)  \coloneqq T(\Q) \cap \SL_N(\Q) \subset \SL_N(\Q)$. We first consider the split torus
\begin{align}
T_0(\Q) &= \left \{ \begin{pmatrix}
T_1 & & \\
 & \ddots & \\
& & T_N
\end{pmatrix}, \, T_m  \in \Q^\times \right \} \simeq (\Q^\times)^{N} \\
T^{(1)}_0(\Q) & = \left \{ \begin{pmatrix}
t_1 & & & \\
 & \ddots & & \\
& & t_{N-1}& \\
& & & (t_1\cdots t_{N-1})^{-1}
\end{pmatrix} \, T_m  \in \Q^\times \right \} \simeq (\Q^\times)^{N-1}.
\end{align}
The change of variables $(T_1, \dots, T_N)=(\frac{t_1}{u}, \dots, \frac{t_{N-1}}{u},\frac{1}{ut_1\cdots t_N})$ is a diffeomorphism $T_0(\R) \simeq T_0^{(1)}(\R) \times \R_{>0}$.
The intersection $T_0(\R) \cap \SO(N)$ is the maximal compact $\{ \pm 1\}^{N-1} \subset T_0(\R)$. Hence the quotient 
\begin{align}
T^{(1)}_0(\R)/T^{(1)}_0(\R) \cap \SO(N) \simeq \R_{>0}^{N-1}
\end{align}
can be embedded in $S$. Let $\vartheta_{T_0} \in \Omega^1(T_0(\R)) \otimes \End(\R^N)$ be the pullback of the Maurer-Cartan $\vartheta$ to $T_0$. Then
\begin{align}
\vartheta_{T_0} = \begin{pmatrix}
dT_1/T_1 & & & \\
 & \ddots & & \\
& & dT_{N-1}/T_{N-1}& \\
& & & dT_N/T_N,
\end{pmatrix}
\end{align}
and $\lambda_{T_0}=\frac{1}{2}(\vartheta_{T_0} +\vartheta_{T_0} ^t)=\vartheta_{T_0} $. Since $(\lambda_{T_0})_{ij}=0$ if $i \neq j$, it follows that
\begin{align} \label{equation460}
\lambda_{T_0}(\sigma)  & = (\lambda_{T_0})_{1\sigma(1)} \wedge \cdots \wedge (\lambda_{T_0})_{N\sigma(N)}
\end{align}
can only be nonzero if $\sigma$ is the identity function $\sigma(i)=i$. For this function, we see that \eqref{equation460} is equal to
\begin{align}
\lambda_{T_0}(\sigma)=\frac{dT_1}{T_1} \wedge \cdots \wedge \frac{dT_N}{T_N}.
\end{align}
The contraction along $u \frac{\partial}{\partial u}$ is 
\begin{align}
\iota_{u \frac{\partial}{\partial u}} \left (\frac{dT_1}{T_1} \wedge \cdots \wedge \frac{dT_N}{T_N} \right )=N \frac{dt_1}{t_1} \wedge \cdots \wedge \frac{dt_{N-1}}{t_{N-1}}.
\end{align}
Hence, the restriction of $E_{\varphi, \chi}(z_1,\tau,s)$ to $\R_{>0}^{N-1}$ is
\begin{align}
E_{\varphi, \chi}(t,\tau,s)= \Lambda_N(s) y^{\frac{s}{2}} \sum_{\substack{\vbf \in  V \\ \vbf \neq 0}} \four_2(\chi)(\vbf)  \frac{\N(\overline{v\tau+w})}{ \lVert t^{-1} (v\tau+w) \rVert^{2N+s}} N\frac{dt_1}{t_1} \wedge \cdots \wedge \frac{dt_{N-1}}{t_{N-1}}
\end{align}
where $t=(t_1,\dots,t_{N-1})$ and
\begin{align}
\lVert t^{-1} (v\tau+w) \rVert^{2N+s}= \left ( \sum_{m=1}^{N-1}\frac{\vert v_m\tau+w_m\vert^2}{t_m^2}+ (t_1\cdots t_{N-1})^2\vert v_N\tau+w_N\vert^2\right )^{N+s/2}.
\end{align}
Note that for a vector $\mu \in \C^N$ we denote its norm by $\N(\mu)=\mu_1 \cdots \mu_N$. 

To compute the integral over $\R_{>0}^{N-1}$ it will be more convenient to change the variables and compute
\begin{align}
\int_{T^{(1)}_0(\R)/T^{(1)}_0(\R) \cap \SO(N)}E_{\varphi, \chi}(z_1,\tau,s) & = \int_{T^{(1)}_0(\R)/T^{(1)}_0(\R) \cap \SO(N)} \int_0^\8 \Theta_{\varphi, \chi}(z,\tau)u^{-s} \nonumber \\
& =   \int_{T_0(\R)/T_0(\R) \cap \SO(N)} \Theta_{\varphi, \chi}(z,\tau)u^{-s}.
\end{align}
Using that 
\begin{align}
\frac{1}{2}\Gamma \left ( N+\frac{s}{2}\right ) \alpha^{-N+\frac{s}{2}}= \int_0^\8 e^{-\alpha u^2} u^{2N+s}\frac{du}{u}
\end{align}
(for $\alpha=\lVert t^{-1} (v\tau+w) \rVert^{2N+s}$) and the change of variable $(T_1, \dots, T_N)=(\frac{u}{t_1}, \dots, \frac{u}{t_N},ut_1\cdots t_N)$ we find that 
\begin{align}
& \frac{1}{2}\Gamma \left ( N+\frac{s}{2}\right )\int_{\R_{>0}^{N-1}}\frac{1}{ \lVert t^{-1} (v\tau+w) \rVert^{2N+s}} N\frac{dt_1}{t_1} \wedge \cdots \wedge \frac{dt}{t_{N-1}} \nonumber \\
& \hspace{6cm} =\prod_{m=1}^N \int_0^\8 e^{-T_m^2 \vert v_m\tau+w_m\vert^2  } T_m^{2+s/N}\frac{dT_m}{T_m} \nonumber \\
& \hspace{6cm} = \frac{1}{2^N} \Gamma\left ( 1 + \frac{s}{N} \right )^N \prod_{m=1}^N \frac{1}{\vert v_m\tau+w_m\vert^{2+s/N}}.
\end{align}
We deduce that
\begin{align}
\int_{\R_{>0}^{N-1}} E_{\varphi, \chi}(z_1,\tau,s)= \Lambda'_N(s) \sum_{\vbf \in V} \four_2(\chi)(\vbf) \frac{y^{\frac{s}{2}}}{ \N(v\tau+w)\lvert \N(v\tau+w)\rvert^{s/N}}
\end{align}
where
\begin{align}
\Lambda'_N(s)=\frac{(-1)^{N-1}i^{N}\Gamma\left ( 1 + \frac{s}{N} \right )^N}{\pi^{N+\frac{s}{2}}}.
\end{align}

\subsubsection{Integral over modular symbols.} Let $L=n\Z^{N}$ for some positive integer $n$, and $\Lcal=L \times L$. Let $\vbf_0=(v_0,w_0)$ where $v_0=(a_1, \cdots, a_N)^t$ and $w_0=(b_1, \cdots, b_N)^t$ are two vectors in $L^\ast = n^{-1}\Z^{N}$. Ash and Rudolph \cite{ashrud} proved that the relative homology of $\overline{S}_\Gamma$ is generated by modular symbols $Z_Q$ attached to an invertible matrix $Q \in \Mat_N(\Z)$ (if $\det(Q) =0$, then the modular symbol is trivial in homology). First, let $Z_\8 \in Z^{\BM}_{N-1}(S_\Gamma;\Z)$ be the image of $T_0^{(1)}(\R)T_0^{(1)}(\R)^t \simeq \R_{>0}^{N-1}$ in $S_\Gamma$. The modular symbol $Z_Q$ is\footnote{It is exactly the set $D$ in \cite[p.~245]{ashrud}.} the image of $S_Q \coloneqq QS_\8Q^t$ in $S_\Gamma$ and defines a class
\begin{align}
Z_Q \in H^{\BM}_{N-1}(S_\Gamma,\Z).
\end{align}

Let $E_1(\tau,\lambda_0,s)$ be the weight one Eisenstein series
\begin{align}
E_1(\tau,\lambda_0,s) \coloneqq \sideset{}{^{\prime}}\sum_{\lambda \in \Z+\tau\Z }\frac{1}{(\lambda_0+\lambda)\vert \lambda_0+\lambda\rvert^{s/N}}
\end{align}
where $\lambda_0 \in \C$ and the $'$ means that we remove $-\lambda_0$ from the summation if $\lambda_0 \in \Z \tau+\Z$. The sum converges for $\re(s) \gg 0$ and admits an analytic continuation to the entire plane (it vanishes at $s=0$ if $\lambda_0 \in \Z\tau+ \Z$).

By the previous computations we find that
\begin{align} \label{prodeisenstein}
\int_{Z_\8} E_{\varphi, \chi}(z_1,\tau,s)& = \Lambda'_N(s) \sum_{\vbf \in V} \four_2(\chi)(\vbf) \frac{y^{\frac{s}{2}}}{ \N(v\tau+w)\lvert \N(v\tau+w)\rvert^{s/N}} \nonumber \\
&= \Lambda'_N(s)\sum_{\vbf_0 \in D_\Lcal} \four_2(\chi)(\vbf_0) \sum_{\substack{v \in v_0+n\Z^N \\ w \in w_0+n\Z^N}} \frac{y^{\frac{s}{2}}}{ \N(v\tau+w)\lvert \N(v\tau+w)\rvert^{s/N}} \nonumber \\
& = \Lambda'_N(s)n^{-N} \sum_{\vbf_0 \in D_\Lcal} \four_2(\chi)(\vbf_0) \prod_{m=1}^N E_1\left (\tau,\frac{a_m\tau+b_m}{n},\frac{s}{N}\right ).
\end{align}

Note that since $Q$ has integral entries, we have $Q\Lcal^\vee \subset \Lcal^\vee$ and $Q\Lcal \subset \Lcal$. Thus, the map $\vbf \mapsto \chi(Q\vbf)$ can be viewed as a test function in $\C[D_\Lcal]$, supported on $Q\Lcal^\vee \subset \Lcal^\vee$ and $\Lcal$-invariant. Using the equivariance of the form $\varphi$, its Fourier transform satisfies
\begin{align}
Q^\ast \four_2(\varphi)(\vbf)=\four_2(\varphi)(Q^{-1}\vbf),
\end{align}
and the following theorem follows.

\begin{thm} The period over the modular symbol $Z_Q$ converges and we have
\begin{align}
\int_{Z_Q} E_{\varphi, \chi}(z_1,\tau)& = \frac{(-1)^{N-1}i^N}{\pi^{N}n^{N}} \sum_{\vbf_0 \in V/\Lcal}\four_2(\chi)(Q\vbf_0) \prod_{m=1}^N E_1\left (\tau,\frac{a_m\tau+b_m}{n} \right ).
\end{align}
\end{thm}

\begin{rmk} For an $N$-tuple $\underline{\gamma}=(\gamma_1, \dots, \gamma_N)$, one can consider similar periods over $Z_{Q(\underline{\gamma})}$. Similar computation show that obtains the homogeneous cocycle valued in meromorphic functions $\HH \times \C^N$ considered in Zhang's thesis \cite[Section.~4.1.1]{haozhang} and by Bergeron-Charollois-Garcia in \cite[Theorem.~2.10]{bcgcrm}.
\end{rmk}

 \subsubsection{Tori attached to totally real fields} \label{real tori} In this section, we compute the period over a cycle attached to a totally real field $F$. See also \cite{rbrsln2} for a more general (and detailled) computation. 
  Let $\Ocal$ be the ring of integers of $F$ and $\sigma_1,\dots,\sigma_N$ the $N$ real embeddings. Let $F^{\times,+}$ be the set of totally positive elements in $F$, and $\Ocal^{\times,+} = \Ocal^{\times} \cap F^{\times,+}$ the totally positive units. Let $\epsilon=(\epsilon_1, \dots, \epsilon_N)$ be a $\Z$-basis of $\Ocal$, chosen such that
\begin{align}
g_{\Ocal} \coloneqq \begin{pmatrix}
\sigma_1(\epsilon_1) & \cdots & \sigma_1(\epsilon_N) \\[1em]
\vdots & & \vdots \\[1em]
\sigma_N(\epsilon_1) & \cdots & \sigma_N(\epsilon_N)
\end{pmatrix} \in \GL_N(\R)
\end{align}
has positive determinant. Let $\m \subset \Ocal$ be an integral ideal, that gives a lattice $L=\m$ in $\Z^N$ with respect to the above basis. Its dual is $L^\ast=\m^{-1}\dfrak^{-1}$, and let $\chi_\m \in \C[D_{\Lcal}]$ be the test function
\begin{align}
\chi_\m=\id_{L \times L^\ast}=\id_{\m \times \m^{-1}\dfrak^{-1}}.
\end{align}
The choice of the basis induces an embedding $R_\Ocal \colon F^\times \hooklongrightarrow \GL_N(\Q)$ via the regular representation, for which $R_\Ocal(\Ocal^{\times,+}) \subseteq \SL_N(\Z)$. For $\nu \in F^\times$ we have
 \begin{align} \label{equiv embed F}
R_\Ocal(\nu)=g^{-1}_{\Ocal}\begin{pmatrix} \sigma_1(\nu) &  & 0 \\  & \ddots &  \\ 0 &  & \sigma_N(\nu) \\ \end{pmatrix}g_{\Ocal},
\end{align}
 so that the image of $R_\Ocal$ is a maximal torus conjugate to $T_0(\R)$.
 
The totally positive elements $F^{1,+}$ of norm in $F$ can be diagonally embedded (via the embeddings)
 \begin{align}
(F\otimes \R)^{1,+} \hooklongrightarrow \SL_N(\R)
\end{align}
and the image is the split torus $T_0^{(1)}(\R)$. Since $\Gamma_\Ocal \coloneqq R_\Ocal^{-1}(R_\Ocal(\Ocal^{\times,+}) \cap \Gamma)$ has finite index in $\Ocal^{\times,+}$, it follows from Dirichlet's unit theorem that $\Gamma_\Ocal \backslash (F\otimes \R)^{1,+}$ is compact.
If $Z_\8 \subset S$ denotes the orbit of $T_0^{(1)}(\R)$ as in the previous section, then we have a diffeomorphism
\begin{align}
\Gamma_\Ocal \backslash (F\otimes \R)^{1,+} \longrightarrow \Gamma \backslash g_\Ocal^{-1} Z_\8 \subset S_\Gamma.
\end{align}
The map is $\Gamma_\Ocal$-equivariant by \eqref{equiv embed F}, and its image defines a cycle
\begin{align}
Z_\Ocal \in Z_{N-1}(S_\Gamma;\Z),
\end{align}
which is compact and defines a homology class in $S_\Gamma$. It follows from the above computations (after conjugating by $g_\Ocal$) that
\begin{align}
\int_{Z_\Ocal} E_{\varphi,\chi_\m}(z_1,\tau,s)= \Escr_\m(\tau, s)
\end{align}
is the diagonal restriction of a Hilbert-Eisenstein series of parallel weight one
\begin{align}
\Escr_\m(\tau, s) = \sideset{}{^{\prime}}\sum_{(v,w) \in \m\times \m/\Gamma_\Ocal}\frac{\N(\im(\tauud))^{\frac{s}{2N}}}{ \N(v\tauud+w)\lvert \N(v\tauud+w)\rvert^{s/N}}.
\end{align}
Here $v\tauud+w$ denotes the vector in $\C^N$ that has entries $v_m \tau_m+w_m$, where $(v,w)=(v_1, \dots,v_N,w_1, \dots,w_N)$.
\begin{thm} \label{diagonal restriction}
The diagonal restriction of the Hilbert-Eisenstein series has the Fourier expansion
\begin{align}
\Escr_{\m}(\tau)=\int_{Z_\Ocal} z(\chi_\m)+(-1)^{\frac{N(N-1)^2}{2}}\sum_{n=1}^\8 (Z_\Ocal,Z_n(\chi_\m) ) q^n.
\end{align}
\end{thm}

\section{The case $N=2$} \label{caseN2}

 Let $X$ be the space of positive definite symmetric $2$ by $2$ matrices and $S \subset X$ the matrices of determinant $1$. We have
 \begin{align}
     \HH \simeq \SL_N(\R)/\SO(2) \simeq S.
 \end{align}
 Let $z_1=a+ib$ be the coordinates on $\HH$. The first map sends $z_1 \in \HH$ to $g_1\SO(2)$ where
 \begin{align}
     g_1=\begin{pmatrix}
         \sqrt{b} & a/\sqrt{b} \\
         0 & 1/\sqrt{b}
     \end{pmatrix}.
 \end{align}
 The second map sends it to the positive definite quadratic form
 \begin{align}
         g_1 g^t_1=\frac{1}{b} \begin{pmatrix}
             a^2+b^2 & a \\ a & 1
         \end{pmatrix} \in S.
 \end{align}
 We extend it to
 \begin{align}
     \HH \times \R_{>0} \longrightarrow X,\qquad
     (z_1 , u) \longmapsto z_1u=\frac{u}{b} \begin{pmatrix}
             a^2+b^2 & a \\ a & 1
         \end{pmatrix}.
 \end{align}
Let $\chi=\id_{l \oplus \Z^2}$ be the characteristic function of the lattice coset $l \oplus \Z^2 \subset V_\Z$, where $l \subset \Z^2$ is the set
\begin{align*}
l=\left \{ \biggl . v=(v_1,v_2)^t \in \Z^2 \  \biggr \vert \ (v_1,p)=1, \ v_2 \in p\Z  \right \}.
\end{align*}
The function $\chi$ is preserved by $\Gamma=\Gamma_0(p)$ and the locally symmetric space is the modular curve
\begin{align}
    S_\Gamma = Y_0(p)=\Gamma_0(p)\backslash \HH.
\end{align}

\subsection{Hecke operators}
For two points $\alpha, \beta \in \PP^1(\Q)$ let $\{\alpha,\beta\}$ be the geodesic joining the two points on the boundary of $\HH$. For $n>0$ we define
\begin{align}
    \Delta_0(p)^{(n)} = \left \{ \left . M=\begin{pmatrix}
       a & b \\ pc & d 
    \end{pmatrix} \in \Mat_2(\Z) \ \right \vert \ \det(M)=n>0, \ (a,p)=1 \right \}.
\end{align}

The congruence subgroup $\Gamma_0(p) \subset \SL_N(\Z)$ acts on the left and the right on $\Delta_0(p)^{(n)}$ and the double coset
\begin{align} \label{doublecosets}
    \Rcal_0(p)^{(n)}=\Gamma_0(p) \backslash \Delta_0(p)^{(n)}
\end{align}
is finite. 
We will also denote by $\{\alpha,\beta\}$ the image of $\{\alpha,\beta\}$ in the modular curve $\Gamma_0(p) \backslash \HH$. The modular symbol $\{\alpha,\beta\}$ represents a $1$-cycle relative to the cusps of $\Gamma_0(p)\backslash \HH$. An explicit set of representatives for $\Rcal_0(p)^{(n)}$ is given by
\begin{align}
 \Delta_0(p)^{(n)} = \bigsqcup_{\substack {  dd'=n \\ d>0  \\ (d,p)=1} } \bigsqcup_{b=0}^{d'-1} \Gamma_0(p) \begin{pmatrix} 
      d & b \\ 0 & d'
  \end{pmatrix}.
\end{align}
The action of the coset representatives induces the Hecke operator
\begin{align}
    T_n\{\alpha,\beta\} \coloneqq \sum_{\substack {  dd'=n \\ d>0  \\ (d,p)=1} } \sum_{b=0}^{d'-1} \left \{  \frac{d\alpha+b}{d'},\frac{d\beta+b}{d'}\right \}.
\end{align}
In particular, we have
\begin{align}
    T_n\{0,\8\} = \sum_{\substack {  dd'=n \\ d>0  \\ (d,p)=1} } \sum_{b=0}^{d'-1}  \left \{ \frac{b}{d'},\8 \right \}.
\end{align}
The Hecke operators act on the homology $H_1(Y_0(p);\Z)$ in a similar way. 
\begin{prop} \label{cycleN2}
    We have $Z_n(\chi)=T_n\{0,\8\}$.
\end{prop}
\begin{proof}
Let $\vbf=(v,w)$ where $v=(v_1,v_2)^t$ and $w=(w_1,w_2)^t$. Let us compute $X_{\vbf}$.
 It is the set of $z=(z_1,u)$ such that $v=z_1uw$ {\em i.e.}
 \begin{align}
     \begin{pmatrix}
         v_1 \\ v_2
     \end{pmatrix}= \frac{u}{b} \begin{pmatrix}
             a^2+b^2 & a \\ a & 1
         \end{pmatrix} \begin{pmatrix}
             w_1 \\ w_2
         \end{pmatrix}= \frac{u}{y}\begin{pmatrix}
             w_1(a^2+b^2)+w_2a \\ w_1a+w_2 
         \end{pmatrix}.
 \end{align}
 Hence the projection $S_{\vbf}$ is the set of $\tau$ such that
 \begin{align} \label{quadeq}
      v_2(w_1(a^2+b^2)+w_2a)&=v_1(w_1a+w_2) \\
      & (=bu^{-1} \ \textrm{for some} \ u \in \R_{>0} ). \nonumber
 \end{align}
 If $w_1=0$, then $v_2 \neq 0$ and $w_2 \neq 0$ since $Q(\vbf)>0$. The equation is $v_2a=v_1$ and describes the geodesic from $\frac{v_1}{v_2}$ to $\8$ in $\HH$. 
 If $v_2=0$, then $v_1 \neq 0$ and $w_1 \neq 0$. The equation is $0=w_1a+w_2$ and describes the geodesic from $-\frac{w_2}{w_1}$ to $\8$.
 If $w_1v_2 \neq 0$, then we can rewrite equation \eqref{quadeq} as
 \begin{align}
     \left ( a+\frac{v_2w_2-v_1w_1}{2w_1v_2} \right )^2+b^2=\left ( \frac{v_2w_2+v_1w_1}{2w_1v_2} \right )^2,
 \end{align}
which describes the geodesic with endpoints $\frac{w_2}{w_1}$ and $-\frac{v_1}{v_2}$. We have showed that the submanifold is the geodesic
    \begin{align}
      X_\vbf=\left \{ \frac{w_2}{w_1},-\frac{v_1}{v_2}\right \} \qquad  \textrm{if} \ \ \vbf =\left ( \begin{pmatrix}
          w_1 \\ w_2
      \end{pmatrix} ,\begin{pmatrix}
          v_1 \\ v_2
      \end{pmatrix}\right ).
    \end{align}
By the proof of Proposition \ref{finiterep}, we find that a set of representatives for $\Gamma_0(p) \backslash (l \times \Z^2)
$ is
\begin{align} \label{representatives}
        R_n = \left \{ \left . \vbf = \left ( \begin{pmatrix}
        d \\ 0
    \end{pmatrix},\begin{pmatrix}
        d' \\ b
    \end{pmatrix} \right ) \right \vert n=dd', \ (d,p)=1, \ 0 <d, \ 0 \leq b \leq d'-1  \right \}.
    \end{align}
 Hence, we have
\begin{align}
    Z_n(\chi)=\sum_{\substack {  dd'=n \\ d>0  \\ (d,p)=1} } \sum_{b=0}^{d'-1} \left \{\frac{b}{d'},\8 \right \}=T_n\{0,\8\}.
\end{align}
\end{proof}

The Hecke operators also act on differentials forms as follows. If $\omega \in \Omega^1(\HH)^{\Gamma_0(p)}$ is a  $\Gamma_0(p)$-invariant form on $\HH$, then
\begin{align}
T_n \omega \coloneqq \sum_{X \in \Rcal_0(p)^{(n)}} X^\ast \omega
\end{align}
is a  $\Gamma_0(p)$-invariant form on $\HH$. It induces an action of Hecke operators on the cohomology group $H^1(Y_0(p))$. If $\omega_{\{\alpha,\beta\}} \in \Omega^1(Y_0(p))$ is a Poincaré dual to  $\{\alpha,\beta\}$, then $T_n\omega_{\{\alpha,\beta\}}=\omega_{T_n\{\alpha,\beta\}}$. Finally, the Hecke operators acts on a modular form $f \in \Mcal_2(\Gamma_0(p))$ by
\begin{align}
T_nf(z) \coloneqq  \sum_{\substack {  dd'=n \\ d>0  \\ (d,p)=1} } \frac{1}{(d')^2} \sum_{b=0}^{d'-1} f\left (\frac{dz+b}{d'} \right ).
\end{align}
If $\omega_{f}=f(z)dz$, then $T_n(\omega_{f})=\omega_{T_n f}$.

\subsection{The theta lift when $N=2$} With the setup as above, the theta lift is
\begin{align}
    E_{\varphi, \chi} \colon H_1(Y_0(p);\C) \simeq H^1_c(Y_0(p);\C) \longrightarrow \Mcal_2(\Gamma_0(p)).
\end{align}
Let $\{E_2^{(p)},f_1, \dots,f_r\}$ be a basis of $\Mcal_2(\Gamma_0(p))$, where $f_i \in \Scal_2(\Gamma_0(p))$ is a basis of normalized newforms and $E_2^{(p)}$ is the Eisenstein series 
\begin{align}
E_2^{(p)}=\frac{p-1}{24}+\sum_{n \geq 1} \sigma_1^{(p)}(n)q^n.
\end{align}

We define $\omega_E \coloneqq E_2^{(p)}(\tau)d\tau$ and $\omega^{\pm}_f \coloneqq (\omega_{f} \pm \overline{\omega}_{f})$ where $\omega_f=f(\tau)d\tau$.

\begin{cor}
For $\omega \in \Omega_c^{1}(Y_0(p);\C)$ we have the Fourier expansion
    \begin{align}
       E_{\varphi, \chi}(\omega,\tau)= (\omega,\omega_E)- \sum_{n=1}^\8( \omega,T_n\{0,\8\} ) q^n \in M_{2}(\Gamma_0(p)).
    \end{align}
and the spectral expansion
\begin{align}
    E_{\varphi, \chi}(\omega) = \frac{24}{p-1} ( \omega , \omega_E ) E_2^{(p)}- \sum_{i=1}^r \frac{L(f_i,1)}{i\pi\lVert f_i \rVert^2}( \omega ,\omega_{f_i}^+ ) f_i.
\end{align}
In particular, we deduce that the image of this lift is precisely the space $\Mcal^0_2(\Gamma_0(p))$ spanned by $E_2^{(p)}$ and the forms $f_i$ for which $L(f_i,1) \neq 0$.
Moreover, the lift is Hecke equivariant {\em i.e.}, satisfies $E_{\varphi, \chi}(T_n\omega) =T_nE_{\varphi, \chi}(\omega)$ when $(p,n)=1$. 
\end{cor}

\begin{proof} The first part follows from Proposition \ref{cycleN2} and the Fourier expansion of $E_{\varphi, \chi}$. The winding element $G_w \in H^1(Y_0(p))$ is such that
\begin{align}
    ( \omega, G_w )=\int_{0}^\8 \omega.
\end{align}
We follow the proof of \cite[Lemma.~3.4]{dpv} with sligthly different normalizations. We write the winding element as 
\begin{align}
G_w=\alpha_E \omega_E+\sum_{i=1}^r (\alpha_{f_i}^+\omega_{f_i}^++\alpha_{f_i}^-\omega_{f_i}^-).
\end{align}
If $\omega_0$ denotes the homology class represented by $\begin{pmatrix}
1 & 1 \\ 0 & 1
\end{pmatrix}$, then we deduce from $[ \omega_0, \omega_E ]=a_0(E_2^{(p)})$ and $( \omega_0, G_w )=1$ that
\begin{align}
\alpha_E=\frac{24}{p-1}.
\end{align}
Using the following properties
\begin{align}\label{properties pairing}
(\omega_{f_i}^+,\omega_{f_j}^+ )&=( \omega_{f_i}^-,\omega_{f_j}^- )=0, \nonumber \\
( \omega_{f_i}^+,\omega_{f_j}^- )&=-( \omega_{f_j}^-,\omega_{f_i}^+ )=-2 \langle f_i, f_j \rangle_\pet \nonumber \\
( \omega_{f}^-,G_w )&=-\frac{L(f,1)}{i\pi},  \\
( \omega_{f}^+,G_w )&=0,
\end{align}
we find that
\begin{align}
\alpha_{f_i}^-=0, \qquad \alpha_{f_i}^+=-\frac{L(f_i,1)}{i\pi\lVert f_i \rVert^2},
\end{align}
and
\begin{align}
    G_w= \frac{24}{p-1} \omega_E-\sum_{i=1}^r \frac{L(f_i,1)}{i\pi\lVert f_i \rVert^2} \omega_{f_i}^+.
\end{align}
Since $f_i$ is a newform and has real Fourier coefficients, the $n$-th Fourier coefficient of $E_{\varphi, \chi}(\omega)$ is
\begin{align}
 ( \omega, T_n G_w ) = \frac{24}{p-1} ( \omega , \omega_E ) a_n(E_2^{(p)})- \sum_{i=1}^r \frac{L(f_i,1)}{i\pi\lVert f_i \rVert^2} ( \omega ,\omega_{f_i}^+ ) a_n(f_i).
\end{align}
Thus, we have
\begin{align}
    E_{\varphi, \chi}(\omega) = \frac{24}{p-1} ( \omega , \omega_E ) E_2^{(p)}- \sum_{i=1}^r \frac{L(f_i,1)}{i\pi\lVert f_i \rVert^2}(\omega ,\omega_{f_i}^+ )f_i.
\end{align}
The surjectivity follows from the fact that 
\begin{align}
E_{\varphi, \chi}(\omega_E)& =E_{2}^{(p)} \nonumber \\
E_{\varphi, \chi}(\omega_{f_i}^-)& =\frac{2L(f_i,1)}{i\pi}f_i  \\
E_{\varphi, \chi}(\omega_{f_i}^+)& =0. \nonumber
\end{align}

Finally, since the Hecke operator $T_n$ is self-adjoint when $(n,p)=1$ \cite[Theorem.~4.5.4]{miyake}, we have 
\begin{align}
(T_n \omega , \omega_E ) & = (\omega , T_n\omega_E ) E_2^{(p)} =a_n(E_2^{(p)}) ( \omega , \omega_E ). 
\end{align}
Similarly, we have $(T_n \omega , \omega_{f_i}^- )= a_n(f)( \omega , \omega_{f_i}^- )$ and this proves the Hecke equivariance.
\end{proof}
\begin{rmk}\label{remark span Tn} It follows from the proof of the the previous corollary that the kernel of $E_{\varphi, \chi}$ is spanned by all the forms $\omega_{f_i}^+$, and the forms $\omega_{f_i}^-$ for which $L(f_i,1) \neq 0$. Moreover, by Theorem \ref{theoremC416} we have
\begin{align}
\ker(E_{\varphi, \chi}) \simeq \Span \{T_n\{0,\8\} \ \vert \ n \in \NN_{>0} \}^\perp \subset H^1_c(Y_0(p);\C).
\end{align}
It follows that
\begin{align}
\Span \{T_n\{0,\8\} \ \vert \ n \in \NN_{>0} \}^\perp=\span \{\omega_{f_i}^+\} \oplus \span \{\omega_{f_i}^- \ \textrm{such that} \ L(f_i,1)\neq 0\}.
\end{align}
This equality can be seen directly. Since the modular symbols generate the relative homology $H_1^{\BM}(Y_0(p);\C) \simeq H_1(X_0(p),\partial X_0(p);\C)$, the Poincaré-Lefschetz pairing is
\begin{align}
H^1_c(Y_0(p);\C) \times H_1(X_0(p),\partial X_0(p);\C)  \longrightarrow \C, \quad \left ( \omega,\{ \alpha,\beta\} \right ) = \int_\alpha^\beta \omega.
\end{align}
Since $H^1_c(Y_0(p),\C)$ is spanned by the newforms $\omega_{f_i}^+,\omega_{f_i}^-$, it follows from the properties \eqref{properties pairing} that $\left ( \omega_{f_i}^+, T_n\{0,\8\}   \right )=0$ and
\begin{align}
\left ( \omega_{f_i}^-, T_n\{0,\8\}  \right ) & = \left (  T_n^\ast \omega_{f_i}^-, \{0,\8\} \right ) \nonumber \\
& = \left ( \omega_{a_n(f_i)f_i}^-,\{0,\8\} \right ) \nonumber \\
& = -\frac{a_n(f_i)L(f_i,1)}{i\pi}.
\end{align}
\end{rmk}
We conclude by proving the following corollary (Corollary  \ref{corollaryG} in the introduction).
\begin{cor} The space $\Mcal_2^{0}(\Gamma_0(p))$ is spanned by diagonal restrictions of Hilbert-Eisenstein series.
\end{cor}
\begin{proof} The homology $H_1(Y_0(p),\C)$ is generated by the images in the quotient of paths $\{z, \gamma z\}$ in $\HH$, as $\gamma$ ranges through parabolic and hyperbolic matrices in $\Gamma_0(p)$ (since elliptic matrices have finite order, their image in homology is trivial). First, we claim that it is generated by paths $\{z, \gamma z\}$ where $\gamma$ is hyperbolic, since any parabolic matrix can be writen as a product of hyperbolic matrices. For the standard parabolic matrix stabilizing $\8$, we can explicitly write
\begin{align}
\begin{pmatrix}
1 & 1 \\ 0 & 1
\end{pmatrix}=\begin{pmatrix}
2p+1 & 2p+3 \\ p & p+1
\end{pmatrix}\begin{pmatrix}
p+1 & -p-2 \\ -p & p+1
\end{pmatrix},
\end{align}
where both matrices on the right-hand side are hyperbolic matrices in $\Gamma_0(p)$. Similarly, we find for the matrix stabilizing $0$
\begin{align}
\begin{pmatrix}
1 & 0 \\ p & 1
\end{pmatrix}=\begin{pmatrix}
p+1 & p+2 \\ p(p+2)& p^2+3p+1
\end{pmatrix}\begin{pmatrix}
p+1 & -p-2 \\ -p & p+1
\end{pmatrix}.
\end{align}
Since $\Gamma_0(p)\backslash \PP^1(\Q)$ is represented by the two classes $0$ and $\8$, any other parabolic matrix is conjugated over $\Gamma_0(p)$ to (the power of) one of the two matrices above. 
Finally, if $\gamma=\begin{pmatrix}
a & b \\ c & d
\end{pmatrix}$ is hyperbolic, then the image of $\{z, \gamma z\}$ is the projection of the geodesic with endpoints the real quadratic points
\begin{align}
\frac{a-d \pm \sqrt{\tr(\gamma)^2-4}}{2c}.
\end{align}
This geodesic can be obtained as a cycle $Z_\Ocal$ in \ref{real tori} by suitably embedding the ring of integers $\Ocal$ of $\Q(\sqrt{\tr(\gamma)^2-4})$ in $\Z^2$. It follows from Theorem \ref{diagonal restriction} that the lift of this cycle is the diagonal restriction of a Hilbert-Eisenstein series for the real quadratic field $\Q(\sqrt{\tr(\gamma)^2-4})$.
\end{proof}

\printbibliography
\end{document}

%% file: Revision-commands2.tex
\newcommand{\iotau}{\iota_{u\frac{\partial}{\partial u}}}

\newcommand{\8}{\infty}

\newcommand{\Z}{\mathbb{Z}}
\newcommand{\NN}{\mathbb{N}}

\newcommand{\id}{\mathbf{1}}

\newcommand{\Obf}{\underline{0}}

\newcommand{\tauud}{\underline{\tau}}
\newcommand{\wud}{\underline{w}}

\newcommand{\vud}{\underline{v}}

\newcommand{\jov}{\overline{j}}

\newcommand{\vbf}{\mathbf{v}}

\newcommand{\bas}{\mathrm{bas}}

\newcommand{\four}{\pazocal{F}}

\newcommand{\Span}{\mathrm{span}}
\newcommand{\SO}{\mathrm{SO}}
\newcommand{\Ad}{\mathrm{Ad}}

\newcommand{\SL}{\mathrm{SL}}
\newcommand{\Sp}{\mathrm{Sp}}

\newcommand{\KM}{\mathrm{KM}}
\newcommand{\BM}{\mathrm{BM}}

\newcommand{\GL}{\mathrm{GL}}

\newcommand{\Util}{\widetilde{U}}

\newcommand{\Gamtil}{\widetilde{\Gamma}}

\newcommand{\so}{\mathfrak{so}}

\newcommand{\D}{\mathbb{D}}

\newcommand{\HH}{\mathbb{H}}
\newcommand{\Escr}{\mathscr{E}}

\newcommand{\R}{\mathbb{R}}

\newcommand{\Q}{\mathbb{Q}}
\newcommand{\C}{\mathbb{C}}

\newcommand{\PP}{\mathbb{P}}

\newcommand{\dfrak}{\mathfrak{d}}
\newcommand{\p}{\mathfrak{p}}
\newcommand{\m}{\mathfrak{m}}

\newcommand{\g}{\mathfrak{g}}

\newcommand{\kfrak}{\mathfrak{k}}

\newcommand{\End}{\mathrm{End}}

\newcommand{\Hom}{\mathrm{Hom}}

\newcommand{\tr}{\mathrm{tr}}

\newcommand{\N}{\pazocal{N}}

\newcommand{\sing}{\mathrm{sing}}

\newcommand{\reg}{\mathrm{reg}}

\newcommand{\Ocal}{\mathcal{O}}
\newcommand{\Scal}{\pazocal{S}}

\newcommand{\Lcal}{\pazocal{L}}

\newcommand{\Tcal}{\pazocal{T}}

\newcommand{\Rcal}{\mathcal{R}}

\newcommand{\re}{\mathrm{Re}}

\newcommand{\Mat}{\mathrm{Mat}}
\newcommand{\im}{\mathrm{Im}}

\newcommand{\hooklongrightarrow}{\lhook\joinrel\longrightarrow}

\newcommand\restr[2]{{
  \left.\kern-\nulldelimiterspace 
  #1 
  \vphantom{\big|} 
  \right|_{#2} 
  }}
 
\newcommand{\etatil}{\widetilde{\eta}}

\newcommand{\vptil}{\widetilde{\varphi}}

\newcommand{\Th}{\mathrm{Th}}
\newcommand{\rd}{\textrm{rd}}
\newcommand{\cv}{\textrm{cv}}
\newcommand{\taut}{\textrm{taut}}

\newcommand{\pet}{\mathrm{pet}}
\newcommand{\image}{\mathrm{im}}

\newcommand{\Mcal}{\pazocal{M}}